\documentclass[a4paper,12pt]{extarticle}
\usepackage{amsmath}
\usepackage{amssymb}
\usepackage{amsthm}

\usepackage{hyperref}
\hypersetup{
    colorlinks,
    citecolor=blue,
    filecolor=blue,
    linkcolor=blue,
    urlcolor=blue
}
\allowdisplaybreaks

\begin{document}
\newtheorem{theor}{Theorem}[section] 
\newtheorem{prop}[theor]{Proposition} 
\newtheorem{cor}[theor]{Corollary}
\newtheorem{lemma}[theor]{Lemma}
\newtheorem{sublem}[theor]{Sublemma}
\newtheorem{defin}[theor]{Definition}
\newtheorem{conj}[theor]{Conjecture}
\theoremstyle{definition}
\newtheorem{Bem}[theor]{Remark}
\newtheorem{ex}[theor]{Example}
\hfuzz2cm

\gdef\mtr#1{#1}
\def\bz{\mbox{\boldmath$\zeta$\unboldmath}}
\def\bzs{\mbox{\boldmath$\zeta'$\unboldmath}}
\def\ba{\mbox{\boldmath$a$\unboldmath}}
\def\odd{{\rm odd}}
\gdef\ra{\rightarrow}
\gdef\Bbb{\bf }
\gdef\P1{{\Bbb P}^{1}_{D}}
\gdef\dbd{dd^{c}}
\gdef\a{\alpha}
\gdef\ca{ch_{g}(\alpha)}
\gdef\rl{{\Lambda}}
\gdef\CT{CT}
\gdef\refeq#1{(\ref{#1})}
\gdef\mn{{\mu_{n}}}
\gdef\zn{{\Bbb Z}/(n)}
\gdef\umn{^{\mn}}
\gdef\lmn{_{\mn}}
\gdef\blb{{\big(}}
\gdef\brb{{\big)}}
\gdef\Hom{{\rm Hom}}
\def\Trs{{\rm Tr_s\,}}
\def\Tr{{\rm Tr\,}}
\def\End{{\rm End}}
\def\eq{equivariant }
\def\Td{{\rm Td}}
\def\ch{{\rm ch}}
\def\torus{{\cal T}}
\def\Proj{{\rm Proj}}
\def\bmn{{R_{n}}}
\def\uexp#1{{{\rm e}^{#1}}}
\def\Spec{{\rm Spec}\,}
\def\Qb{\mtr{\Bbb Q}}
\def\Cn{{\bf C}_n}
\def\Zn{{\Bbb Z}/n}
\def\mod{{\rm mod}}
\def\ac1{\ar{\rm c}_{1}}
\def\boxtimes{{\otimes_{\rm Ext}}}
\def\Qmm{{{\Bbb Q}(\mu_{m})}}
\def\NIm#1{{\rm Im}(#1)}
\def\NRe#1{{\rm Re}(#1)}
\def\rk{{\rm rk}\,}
\def\phi{\varphi}
\def\deg{{{\rm deg}\,}}
\def\Td{{\rm Td}}
\def\ch{{\rm ch}}
\def\Hom{{\rm Hom}}
\def\Bil{{\rm Bil}}
\def\End{{\rm End}}
\def\Sym{\Sym}
\gdef\endProof{\hfill$\square$\par\bigskip\noindent}
\def\Im{{\rm Im}\,}
\def\Re{{\rm Re}\,}
\def\im{{\rm im}\,}
\def\ker{{\rm ker}\,}
\def\Tr{{\rm Tr}\,}
\def\Eig{{\rm Eig}\,}
\def \x{\times}
\def \BN{{\mathbb N}}
\def \BZ{{\mathbb Z}}
\def \BQ{{\mathbb Q}}
\def \BR{{\mathbb R}}
\def \BC{{\mathbb C}}
\def \BH{{\mathbb H}}
\def\BP{{\mathbb P}}
\def\Sp{{\bf Sp}}
\def\GL{\mbox{\bf GL}}
\def\SL{\mbox{\bf SL}}
\def\SO{\mbox{\bf SO}}
\def\O{\mbox{\bf O}}
\def\SU{\mbox{\bf SU}}
\def\so{{\frak {so}}}
\def\Aut{\mbox{\rm Aut}}
\def\sign{\mbox{\rm sign\,}}
\def\ord{{\rm ord}}
\def \<{\langle}
\def \>{\rangle}
\def \Ad {{\rm Ad}}
\def \ad {{\rm ad}}
\def\vp{\varphi}
\def\vep{\varepsilon}
\def\theta{\vartheta}
\def\Ric{{\rm Ric}}
\def\vol{{\rm vol}}
\def\diam{{\rm diam}\,}
\def\xx{{\bf x}}
\def\Pf{{\rm Pf}}
\def\grad{{\rm grad}\,}
\def\E{E}
\def\F{{\bf F}} \def\eps{\varepsilon}
\def\div{{\rm div\,}}
\def\divv{{\nabla^*}}
\def\Sym{{\rm Sym}}
\def\Cas{{\rm Cas}}
\def\f{b}
\def\top{{\rm top}}
\def\J{J}
\def\L{{\cal L}}
\def\M{N}
\def\la{a}
\def\ms{n}

\def\bz{\mbox{\boldmath$\zeta$\unboldmath}}
\def\bzs{\mbox{\boldmath$\zeta'$\unboldmath}}
\def\odd{{\rm odd}}

\author{
Kai K\"ohler\footnote{Mathematisches Institut/
Universit\"atsstr. 1/
Geb\"aude 25.22/
D-40225 D\"usseldorf/
koehler@math.uni-duesseldorf.de}} 
\title{The full asymptotic expansion of analytic torsion on homogeneous spaces}
\maketitle
\begin{abstract}
The full asymptotic expansion of the equivariant complex Ray-Singer torsion for high powers of line bundles on symmetric spaces is given in an explicit form. In the case of isolated fixed points this expansion is given for general complex homogeneous spaces. Furthermore the full asymptotic expansion is given for the complex analytic torsion form associated to fibrations by projective curves. The expansions are compared with results by Bismut--Vasserot, Finski and Puchol. The results are applied to lattice representations of Chevalley groups.
\end{abstract}
\begin{center}
\end{center}
\thispagestyle{empty}
\setcounter{page}{1}
\tableofcontents

\section{Introduction}

Ray-Singer's complex analytic torsion $T(M,E)\in\BR$ is an invariant associated to Hermitian holomorphic vector bundles $E$ on compact Hermitian manifolds $M$. It is defined using regularised determinants of Laplace operators. Its primary application lies in the construction of the determinant of a direct image of Hermitian vector bundles within the context of Arakelov geometry. Analogous to the role played by the Kodaira vanishing theorem in algebraic geometry, describing the asymptotic behaviour of analytic torsion for increasing tensor powers $L^\ell$ of positive line bundles $L\to M$ implies arithmetic Hilbert-Samuel theorems in Arakelov geometry. The oldest such result \cite[Th. 8]{BV} was proven by Bismut and Vasserot, who computed the first two leading terms of the asymptotic expansion of the torsion in terms of curvatures and characteristic classes. Gillet and Soul\'e combined this result with their arithmetic Theorem of Riemann--Roch \cite[Th. 7]{GSarr} to obtain the first version of an arithmetic Hilbert-Samuel theorem \cite[Th. 8]{GSarr}. Since then, arithmetic Hilbert-Samuel theorems have played a fundamental and crucial role in Arakelov theory, emphasising the importance of the asymptotic expansion as a natural approach to them.

Bismut and Vasserot extended their result to symmetric powers of positive vector bundles of higher rank (\cite{BV2}). Finski proved that the full asymptotic expansion of $T(M,L^\ell)$ consists only of terms of the form $\ell^k$ and $\ell^k \log\ell$ and he gave a formula for the 3rd and 4th term in the asymptotic expansion (\cite{Fi}). He extended this result to orbifolds. Other extensions and applications to physics have been given by Berman (\cite{Ber}), Su (\cite{Su}) and Larra\'in-Hubach (\cite{La}).

Ray-Singer's analytic torsion has a form-valued generalisation to fibrations, the complex analytic torsion form (\cite{BK}) which can be used to define a full direct image of Hermitian vector bundles in Arakelov geometry. In \cite{Puchol} Puchol generalised Bismut--Vasserot's result to torsion forms and more general bundles than the symmetric powers considered in \cite{BV2}, given as another direct image. Puchol's main result has been extended to equivariant torsion forms by Te\ss mer (\cite{T}, preprint) for high powers of line bundles.

For the real analytic torsion the problem was studied by Bergeron and Venkatesh on locally symmetric spaces (\cite{BV}) and by Bismut, Ma, and Zhang for real analytic torsion forms (\cite{BMZ}). These have been further investigated by Liu (\cite{Liu1}, \cite{Liu2}) for locally symmetric spaces and Q. Ma (\cite{Mq}).

There is an equivariant version $T_t(M,E)$ of complex analytic torsion defined by the author in \cite{K1}. Consider a compact Lie group $G$ and a complex homogeneous space $G/K$. Let $L:=L_{\rho_K+\lambda}\to G/K$ be an equivariant line bundle associated to a weight $\lambda$, where $\rho$ is the Weyl vector, and choose $t\in G$. One goal of the present paper is to give an explicit formula for all terms in the asymptotic expansion of $T_t(G/K,\mtr L^{\ell})$ in term of roots and Weyl chambers, if $t$ acts with isolated fixed points. 

\begin{defin}\label{zetama}
Let $\Phi$ denote the Lerch zeta function and consider a function $P:\BZ\to\BC$ of the form
\begin{equation*}
P(k)=\sum_{j=0}^r c_jk^{\ms_j} e^{ik\phi_j}
\end{equation*}
with $r\in\BN_0$, $\ms_j\in\BN_0$, $c_j\in\BC$,
$\phi_j\in\BR$ for all $j$. Then for $p\in\BR$ we shall use the notations
\begin{eqnarray*}
\bz P&:=\sum_{j=0}^r c_je^{i\vp_j}\Phi(e^{i\vp_j},-\ms,1),\qquad\qquad\qquad\quad
\bzs P&:=\sum_{j=0}^r c_je^{i\vp_j}\frac\partial{\partial s}\Phi(e^{i\vp_j},-\ms,1)
\end{eqnarray*}
and $\bz_{m,a}\sum_{j=0}^rc_je^{ik\vp_j}:=\sum_{j=0}^mc_je^{i\vp_j}\Phi(e^{i\vp_j},-m,a+1)$.
\end{defin}

The following result is proved in Section \ref{SectionHomom}, using characters $\chi_{\rho+\ell\lambda+k\alpha}(t)$ of $G$-representations which for generic $t\in G$ are linear combinations of terms of the form $e^{i(\vp'\ell+\vp k)}$ with $\vp,\vp'\in\BR\setminus2\pi\BZ$. While the general results about complex analytic torsion cited above lack error estimates, in the case considered here the error term takes a very explicit form in Th.\ \ref{isolFixpkt} and error bounds of various levels of precision and simplicity are shown in Lemma \ref{errortermIsol}.
\begin{theor}\label{isolFixpkt1}
Let $M=G/K$ be a compact complex homogeneous space and let $\mtr L:=\mtr L_{\rho_K+\lambda}$ be a $G$-invariant holomorphic Hermitian line bundle on $M$. Assume that $\mtr L$ is positive in the sense that $\forall\alpha\in\Psi:\<\alpha^\vee,\lambda\>>0$. Fix a K\"ahler metric $g_{X_0}$ on $M$. Let $t\in G$ act on $M$ with isolated fixed points. Then the asymptotic behaviour of the equivariant analytic torsion for high powers of $\mtr L$ is given by
\begin{align*}
\lefteqn{
T_t((M,g_{X_0}),\mtr L_{\rho_K+\lambda}^\ell)
=
-\log\ell\cdot\bz\sum_{\alpha\in\Psi}\chi_{\rho+\ell\lambda+k\alpha}(t)
+\bzs\sum_{\alpha\in\Psi}\chi_{\rho+\ell\lambda+k\alpha}(t)
}\\
&-\sum_{\alpha\in\Psi} \log\frac{\langle\a^\vee,\lambda\rangle}{\alpha^\vee(X_0)}\cdot\bz
\chi_{\rho+\ell\lambda+k\alpha}(t)
+\tilde C\cdot
\chi_{\rho+\ell\lambda}(t)
\\
&+\sum_{\a\in\Psi}\sum_{m=1}^{\M-1}\frac{\bz_{m,\langle\a^\vee,\rho\rangle}\chi_{\rho+\ell\lambda+k\alpha}(t)}{(-\langle\a^\vee,\ell\lambda\rangle)^{m}m}
+R_1
\end{align*}
with $|R_1|<\frac{(\M-1)!}{(2\pi \ell c_2)^\M}c_1$ for explicitly given $c_1,c_2\in\BR^+$ which are independent of $\M$ and $\ell$.
The constant $\tilde C$ vanishes if $G$ does not contain a factor of type $G_2$, $F_4$ or $E_8$.
\end{theor}

This is an asymptotic expansion in the following ways: Firstly, for each positive invariant line bundle on a given homogeneous space there exist $r\in\BN, \vp_j,a_{j,m}\in\BR$ for each $j=1,\dots,r,m\in\BZ$ such that Th.\ \ref{isolFixpkt1} provides an expansion in terms of an asymptotic scale
\[
(a_{1,m}e^{i\ell\vp_1}+\cdots+a_{r,m}e^{i\ell\vp_r})\ell^{m}\log\ell,\ (a_{1,m}e^{i\ell\vp_1}+\cdots+a_{r,m}e^{i\ell\vp_r})\ell^{m}
\qquad(m\in\BZ, m\leq m_0).
\]
Secondly, for the summands with $m>0$ Th.\ \ref{isolFixpkt1} can also be regarded as a generalised asymptotic expansion in terms of the asymptotic scale $\ell^{m}\log\ell,\ell^{m}$, where the coefficients are not constant with respect to $\ell$, but bounded. Such a notion is meaningless for the terms with  $m\leq0$ though. Thirdly,
the crucial interpretation for the applications to Arakelov geometry stems from the fact that, in this case, the finitely many angles $\vp_j$ lie in $\BQ\cdot\pi$. Hence there exists a $q\in\BZ^+$ such that they have the form $2\pi \frac{p_j}{q}$. The isotypical components $e^{2\pi i \ell p/q}$ of the action of $\BZ/q\BZ$ thus have an asymptotic expansion in terms of the asymptotic scale $\ell^{m}\log\ell,\ell^{m}$.

When $G/K$ is an Hermitian symmetric space an asymptotic expansion is given for any $t\in G$ in Section \ref{SectionAsymptSymm}, where the characters now are linear combination of terms of the form $k^{m}\ell^{m'}e^{i(\vp'\ell+\vp k)}$ with $\vp,\vp'\in\BR$, $m,m'\in\BN_0$. In the following, $\log^\ddagger(x+1):=-\sum_{k=1}^{\infty}\frac{(-x)^k}k$ is considered as a formal power series. 
\begin{theor}\label{AsymptSymmetricSpaces1}
Consider the same situation as in Th.\ \ref{isolFixpkt1} and additionally assume that G/K is symmetric.
Then the asymptotic behaviour of the equivariant analytic torsion for high powers of $\mtr L$ with respect to an action by any $t\in G$ is given by 
\begin{align*}
T_t(G/K,\mtr L^\ell_{\rho_K+\ell\lambda})&=
-\sum_{\a\in\Psi}\log\frac{\langle\a^\vee,\ell\lambda\rangle}{\a^\vee(X_0)}\cdot\bz\chi_{\rho+\ell\lambda+k\alpha}(t)
+\sum_{\a\in\Psi} \bzs
\chi_{\rho+\ell\lambda+k\a}(t)
\\&
-\sum_{\a\in\Psi}\bz\left(\chi_{\rho+\ell\lambda+k\alpha}(t)
\cdot\log^\ddagger\left(\frac{\langle\a^\vee,\rho\rangle+k}{\langle\a^\vee,\ell\lambda\rangle}+1\right)\right)^{[\deg_\ell>-\M]}
+R_2
\end{align*}
with $|R_2|<\frac{(\dim(G/K)^t+\M-1)!}{(2\pi\ell c_2)^\M}c_1$ for explicitly given $c_1,c_2\in\BR^+$ which are independent of $\M$ and $\ell$.
\end{theor}
As detailed in \cite[Cor. 7.3]{KK} by Kaiser and the author, one term in the formula for the torsion corresponds to one side of the Jantzen sum formula (\cite[p.\ 311]{Jan}), classifying lattice representations of Chevalley schemes. Using this relation Theorems \ref{isolFixpkt1} and \ref{AsymptSymmetricSpaces1} directly yield asymptotic formulae (Th.\ \ref{Jantzen}) for the arithmetic character given by the Jantzen sum formula.

A fourth result concerns the Lie-algebra-equivariant torsion with respect to the action of a vector field introduced by Bismut and Goette (\cite{BG}). Theorem \ref{AsymptP1Lie} provides the full asymptotic expansion of this torsion on $\BP^1\BC$.

A fifth main Theorem \ref{asympTpi} gives the full asymptotic expansion of torsion forms associated to fibrations by rational curves.

Theorem \ref{BVFinski} shows that in the non-equivariant case the four leading terms are equal to the values given by Bismut--Vasserot and Finski. Lemma \ref{VglPuchol} verifies that in the case of a fibration by rational curves, the top terms in the expansion of the torsion form equal the formula presented by Puchol.

Curiously, the comparison with a multiple of $\frac{7}{24}$ in Finski's result requires several unusual results about Lie algebras in Prop.\ \ref{LinWeyldim}-\ref{VglSymm2}, such as the quotient of the Killing forms of $G$ and $K$ being given by $\frac{3\dim K-\dim G}{2(\dim K-1)}$ for a symmetric space $G/K$. Similarly suprising is the fact that the deduction of the asymptotic expansion for the torsion form needs a closed formula for the sum $\sum_{r=0}^j\left({j\atop r}\right)\frac{a^r}{j+m-r}$ (Prop.\ \ref{nichtHa}) that has not been published before. The computation of the asymptotic expansion relies on formulae for the analytic torsions that were given by the author in terms of Lie algebras and the Lerch zeta function in \cite{K2} and \cite{KSgenus}. This reduces large parts of the proofs to results about such objects. Several of the intermediate lemmata are generalisations to the equivariant case of results in \cite[Section 8]{KK}, where they were needed to compute arithmetic heights.

\section{Asymptotic expansion of a derivative of the Lerch zeta function}
The Lerch zeta function is defined by 
$$
\Phi(z,s,v)=\sum_{k=0}^\infty\frac{z^k}{(k+v)^s}
$$
for $z\in\BC$, $|z|<1$, $v\in\BC\setminus\BZ_0^-$ (\cite[1.11(1)]{Erd1}) and $s\in\BC$, ${\rm Re\,}s>1$. As in \cite[p.\ 161]{Ap}, we consider the extension of this function to $|z|\leq1$ and for fixed $z$ its analytic continuation to $s\in\BC$ (for $z\neq1$) and $s\in\BC\setminus\{1\}$ (for $z=1$). Notice that for fixed $\Re s<0$ and any fixed $v$, $\Phi$ is not continuous at $z=1$, as $\lim_{z\to 1\atop z\neq1}|\Phi(z,s,v)|=\infty$.
By definition one obtains for $\Re v>0$\begin{equation}\label{Apostol0}
z\Phi(z,-s,v+1)=\Phi(z,-s,v)-v^s.
\end{equation}
(see \cite[Eq.\ (3.3)]{Ap}). One thus finds
\begin{eqnarray}\label{FundShift}
\Phi(z,s,v+\ell)&=&z^{-\ell}\left(\Phi(z,s,v)-\sum_{k=0}^{\ell-1}\frac{z^k}{(v+k)^s}\right).
\end{eqnarray}

Let $\zeta(s)$ denote the Riemann zeta function and let $\zeta(s,a)=\sum_{m=0}^\infty(m+a)^{-s}$ (for ${\rm Re\, }s>>0$) denote the Hurwitz zeta function. We define the Harmonic Numbers by ${\cal H}_m:=\sum_{j=1}^m\frac1m$, the Bernoulli polynomials via $\frac{te^{xt}}{e^t-1}=:\sum_{m=0}^\infty \frac{t^m}{m!}B_m(x)$ and the Bernoulli numbers as $B_m:=B_m(0)$.

We shall need the following Theorem about asymptotic expansions of a derivative of the Lerch zeta function, as it yields the asymptotic expansion of sums of the form
\begin{eqnarray}\label{SumToPhi}
\sum_{k=1}^{v}k^me^{ik\vp}\log k=
e^{i(v+1)\vp}\frac{\partial\Phi}{\partial s}(e^{i\vp},-m,v+1)
-e^{i\vp}\frac{\partial\Phi}{\partial s}(e^{i\vp},-m,1)
\end{eqnarray}
which appear in the formula for the torsion on symmetric spaces.
Such an asymptotic expansion has been given by Katsurada (\cite[Th. 16]{Kat}). In contrast to his result, we specialise to $s\in\BZ_0^-$ to get more explicit expressions for the coefficients and the error term. For $e^{i\vp}=1$ such a result has been obtained by Elizalde (\cite{Eli}), and we derive an explicit error bound $R$ in the proof.
\begin{theor}\label{AsymptLerch}
Consider $\vp\in\BR$, $\M\in\BZ^+$ and $a\in\BN_0$.
If $e^{i\vp}\neq1$, $s=-\ms\in\BZ_0^-$, $\frac{\partial\Phi}{\partial s}$ has the asymptotic expansion for $v\to\infty$
\begin{eqnarray*}
\frac{\partial\Phi}{\partial s}(e^{i\vp},-\ms,v+\la)&=&
\sum_{m=0}^{\ms}v^{m}\left({\ms\atop m}\right)\Phi(e^{i\vp},m-\ms,\la)\cdot\left(-\log v+{\cal H}_{m}-{\cal H}_{\ms}\right)
\\&&+\sum_{m=1}^{\M-1}\frac{\Phi(e^{i\vp},-m-\ms,\la)}{(-v)^{m}m\left({m+\ms\atop \ms}\right)}
+\frac{R(e^{i\vp},-\M-\ms,\la,\ms,v)}{(-v)^{N}N\left({N+\ms\atop \ms}\right)}.
\end{eqnarray*}
In the case  $e^{i\vp}=1$ one obtains
\begin{eqnarray*}
\frac{\partial\Phi}{\partial s}(1,-\ms,v+\la)&=&v^{\ms+1}\left(\frac{\log v}{\ms+1}-\frac{1}{(\ms+1)^2}\right)
\\&&+\sum_{m=0}^{\ms}v^{m}\left({\ms\atop m}\right)\Phi(1,m-\ms,\la)\cdot\left(-\log v+{\cal H}_{m}-{\cal H}_{\ms}\right)
\\&&+\sum_{m=1}^{\M-1}\frac{\Phi(1,-m-\ms,\la)}{(-v)^{m}m\left({m+\ms\atop \ms}\right)}
+\frac{R(1,-\M-\ms,\la,\ms,v)}{(-v)^{N}N\left({N+\ms\atop \ms}\right)}.
\end{eqnarray*}
In both cases the coefficient of $\log v$ on the right hand side is equal to $-\Phi(e^{i\vp},-\ms,v+\la)\log v$. There is an explicit bound
\begin{eqnarray*}
|R(e^{i\vp},-\M,\la,\ms,v)|&<&\frac{N!}{(2\pi)^{N+1}}\left(\zeta(N+1,\frac{\vp}{2\pi})+\zeta(N+1,1-\frac{\vp}{2\pi})\right)
+\zeta(-N)-\zeta(-N,\la)
\\&=:&C(e^{i\vp},-\M,\la).
\end{eqnarray*}
\end{theor}
Note that there is no $v^{\ms}$-term in this expansion, although there is a $v^{\ms}\log v$-term. The case $e^{i\vp}=1$, $a\in\{0,1\}$ can be simplified as $\Phi(1,-m,0)=\Phi(1,-m,1)=\zeta(-m)$ vanishes for $m$ even. In the special case $\frac{\partial \Phi}{\partial s}(1,0,v)=\log\Gamma(v)-\log\sqrt{2\pi}$ (\cite[eq. 1.10(10)]{Erd1}) Th.\ \ref{AsymptLerch} corresponds to Stirling's approximation, and in general we proceed similarly to \cite[eq. 1.18(9)]{Erd1}.

The term $\zeta(-N)-\zeta(-N,\la)$ is a Faulhaber polynomial $\frac{B_{N+1}(\la)-B_{N+1}}{N+1}=\sum_{m=1}^{a-1}m^N$ for $a\geq1$, and vanishing for $a\in\{0,1\}$. For $\vp\in]0,2\pi[$ one finds
$$\left|\zeta(N+1,\frac{\vp}{2\pi})-(\frac\vp{2\pi})^{-N-1}-(1+\frac\vp{2\pi})^{-N-1}\right|<\int_1^\infty\frac{dk}{(k+\frac\vp{2\pi})^{N+1}}=\frac{1}{N}(1+\frac\vp{2\pi})^{-N}$$by replacing the sum with an integral. In the case $\M$ odd Lerch's transformation formula shows that the first summand in the bound equals $i^{\M+1}e^{i\vp}\Phi(e^{i\vp},-\M,1)$. For $e^{i\vp}\neq1$ this term equals $-\frac12\frac{\partial^\M}{\partial \vp^\M}\cot\frac\vp2$ (\cite[Eqs.\ 1.16(4), 1.16(9)]{Erd1}).
In the case $N=4\tilde N+1$ and $e^{i\vp}=1$, this causes the first summand in the bound to cancel with the second summand $\zeta(-N)$. In the case $N=4\tilde N+3$ and $e^{i\vp}=1$, these terms are equal. 

By \cite[eq. 1.11(15)]{Erd1}, $\Phi(e^{i\vp},-m,0)=\frac{m!}{(-i\vp)^{m+1}}+O(1)$ for $\vp\to0^+$.

\begin{proof}
First notice that
\begin{eqnarray}\label{AblBinomial}
\nonumber
\frac\partial{\partial s}_{|s=-\ms}\left({-s\atop q}\right)&=&\frac{-1}{q!}\sum_{j=0}^{q-1}\prod_{r=0\atop r\neq j}^{q-1}(\ms-r)
\\&=&\left\{
\begin{array}{cl}
  \frac{\ms!(q-1-\ms)!(-1)^{q-\ms}}{q!}=\frac{(-1)^{q-\ms}}{(q-\ms)\left({q\atop \ms}\right)}  &  \mbox{ if }\ms\in\{0,\dots,q-1\}, \\
  \left({\ms\atop q}\right)\sum_{j=0}^{q-1}\frac{1}{j-\ms}   &   \mbox{ else.}
\end{array}
\right.
\end{eqnarray}

For $z\neq1$ and any $\la\in\BN_0$, the value $\Phi(z,-m,\la)$ is a rational function in $z$ given recursively by
\begin{equation}\label{IterationPhi}
\Phi(z,0,\la)=\frac1{1-z},\quad \Phi(z,-m-1,\la)=z^{1-\la}\frac\partial{\partial z}(z^\la\Phi(z,-m,\la)).
\end{equation}
In particular for $m\in\BZ^+$ one finds
\begin{equation}\label{MultipleDiffLerchPhi}
\frac{\partial^m}{\partial u^m}\frac{e^{-\la u}}{1-ze^{-u}}
=z^{-\la}\frac{\partial^m}{\partial u^m}e^{-\la u}z^\la\Phi(ze^{-u},0,\la)
=(-1)^m e^{-\la u}\Phi(ze^{-u},-m,\la).
\end{equation}
For $N\in\BZ^+$, $|z|=1$, $z=e^{i\vp}\neq1$, $u\in[0,\infty]$, Taylor's Theorem with the Lagrange form of the remainder shows the existence of $\xi\in[0,u]$ such that
\begin{eqnarray}\label{weirdTaylor2}
\nonumber
\lefteqn{
\frac{e^{-\la u}}{1-ze^{-u}}-\sum_{m=0}^{\M-1}\frac{(-u)^m\Phi(z,-m,\la)}{m!}
}
\\&=&\frac{u^N}{N!}\frac{\partial^N}{\partial u^N}_{|u=\xi}\frac{e^{-\la u}}{1-ze^{-u}}
=\frac{u^N}{N!}(-1)^Ne^{-\la\xi}\Phi(ze^{-\xi},-N,\la).
\end{eqnarray}
This is Apostol's Taylor expansion in $u$ (\cite[Eq.\ (3.1) and the eq. before]{Ap}). Its radius of convergence is given by $\min\{|u|\,|\,1-ze^{-u}=0\}=\min|\vp+2\pi \BZ|$. Lerch's Transformation formula in its form in \cite[Eq.\ (9)]{Ob}
\begin{eqnarray}
\Phi(e^{i\vp-\xi},-\tilde m,\tilde a)=\Gamma(1+\tilde m)e^{\tilde a(\xi-i\vp)}\sum_{k\in\BZ}\frac{e^{2\pi ik\tilde a}}{(2\pi ik+\xi-i\vp)^{\tilde m+1}}
\end{eqnarray}
($\tilde a\in]0,1]$, $\vp\in[0,2\pi[$, $\Re \tilde m>0$)
shows for $z=e^{i\vp}\neq1$, $\xi>0$, $N\in\BZ^+$ that
\begin{eqnarray}
\nonumber
\lefteqn{
\left|(-1)^Ne^{-\la\xi}\Phi(ze^{-\xi},-N,\la)\right|
\stackrel{(\ref{FundShift})}=\left|e^{i\vp-\xi}\Phi(ze^{-\xi},-N,1)-\sum_{m=0}^{\la-1}e^{m(i\vp-\xi)}m^N\right|
}
\\
\nonumber
&\leq&\frac{N!}{(2\pi)^{N+1}}\left|e^{-i\la\vp}\sum_{k=-\infty}^{+\infty}(k+\frac{\xi-i\vp}{2\pi i})^{-N-1}
\right|+\sum_{m=0}^{\la-1}m^N
\\&\leq& \frac{N!}{(2\pi)^{N+1}}\sum_{k=-\infty}^{+\infty}|k-\frac{\vp}{2\pi}|^{-N-1}
+\zeta(-N)-\zeta(-N,\la)
\nonumber
\\&=&\frac{N!}{(2\pi)^{N+1}}\left(\zeta(N+1,\frac{\vp}{2\pi})+\zeta(N+1,1-\frac{\vp}{2\pi})\right)
+\zeta(-N)-\zeta(-N,\la)
\nonumber
\\&=&C(z,-\M,\la).
\end{eqnarray}
Hence
\begin{eqnarray}\label{weirdTaylor}
\left|\frac{e^{-\la u}}{1-ze^{-u}}-\sum_{m=0}^{\M-1}\frac{(-u)^m\Phi(z,-m,\la)}{m!}\right|u^{-\M}
<\frac{C(z,-\M,\la)}{\M!}.
\end{eqnarray}

The Lerch zeta function verifies for $\Re v>0$ and either $|z|\leq1$, $z\neq1$, $\Re s>0$ or $z=1$, $\Re s>1$ the relation
$
\Phi(z,s,v)=\frac1{\Gamma(s)}\int_0^\infty\frac{u^{s-1}}{1-ze^{-u}}\frac{du}{e^{vu}}
$ (\cite[eq. 11.1(3)]{Erd1}). 
One obtains for $z\neq1$, $\Re s>0$, $\Re v>0$
\begin{eqnarray}\nonumber\label{PhiIntegral1}
\Phi(z,s,v+\la)&=&\frac1{\Gamma(s)}\int_0^\infty\left(\frac{e^{-\la u}}{1-ze^{-u}}-\sum_{m=0}^{\M-1}\frac{(-u)^m\Phi(z,-m,\la)}{m!}\right)u^{s-1}\frac{du}{e^{vu}}
\\&&+\sum_{m=0}^{\M-1}\frac{(-1)^m\Phi(z,-m,\la)}{m!}\frac{v^{-m-s}\Gamma(m+s)}{\Gamma(s)}
\nonumber
\\&=&\frac1{\Gamma(s)}\int_0^\infty\left(\frac{e^{-\la u}}{1-ze^{-u}}-\sum_{m=0}^{\M-1}\frac{(-u)^m\Phi(z,-m,\la)}{m!}\right)u^{s-1}\frac{du}{e^{vu}}
\nonumber
\\&&+\sum_{m=0}^{\M-1}\Phi(z,-m,\la)v^{-m-s}\left({-s\atop m}\right).
\end{eqnarray}

According to Ineq.\ (\ref{weirdTaylor}) the integral in Eq.\ (\ref{PhiIntegral1}) is holomorphic in $s$ for $\Re s>-\M, v>0$. As the last summand in Eq.\ (\ref{PhiIntegral1}) is holomorphic in $s\in\BC$, by analytic continuation Eq.\ (\ref{PhiIntegral1}) holds for $\Re s>-\M, v>0$. Setting $s=-\ms$ in Eq.\ (\ref{PhiIntegral1}) proves the statement about the coefficient of $\log v$ in the theorem, as the factor $\frac1{\Gamma(s)}$ causes the integral term to vanish.

Using $\left(\frac1{\Gamma(s)}\right)'_{|s=-\ms}=(-1)^{\ms}\ms!$, the derivative at $s=-\ms>-\M$ is given by
\begin{eqnarray}\nonumber\label{asymptPhi1}
\frac\partial{\partial s}\Phi(z,-\ms,v+\la)&=&(-1)^{\ms}\ms!
\int_0^\infty\left(\frac{e^{-\la u}}{1-ze^{-u}}-\sum_{m=0}^{\M-1}\frac{(-u)^m\Phi(z,-m,\la)}{m!}\right)\frac{du}{u^{\ms+1}e^{vu}}
\\&&+\sum_{m=0}^{\M-1}\Phi(z,-m,\la)\frac\partial{\partial s}_{|s=-\ms}\left(v^{-m-s}\left({-s\atop m}\right)\right).
\end{eqnarray}
Let $(-v)^{\ms-N}\frac{R(z,-\M,\la,\ms,v)}{(N-\ms)\left({N\atop \ms}\right)}$ denote the first summand on the right hand side. By using inequality (\ref{weirdTaylor}) when integrating over $u$ one finds
$
|R(z,-\M,\la,\ms,v)|<C(z,-\M,\la).
$
For $s=-\ms$, substituting Eq.\ (\ref{AblBinomial}) in the third summand in Eq.\ (\ref{asymptPhi1}) proves the first formula in the theorem. In the case $z=1$ the result follows by using the classical Taylor expansion
\begin{eqnarray}\label{nonweirdTaylor2}
\frac{e^{-\la u}}{1-e^{-u}}=\frac1u+\sum_{m=0}^\infty\frac{(-u)^m\Phi(1,-m,\la)}{m!}.
\end{eqnarray}
Here $\Phi(1,-m,\la)=\zeta(-m,\la)=-\frac{B_{m+1}(\la)}{m+1}$ holds.
Expanding $\frac{e^{-\la u}}{1-e^{-u}}-\frac1u$ as in (\ref{weirdTaylor2}), one finds the Lagrange remainder estimate
\begin{eqnarray}
\nonumber
\lefteqn{
\left|e^{-\la\xi}\Phi(e^{-\xi},-\M,\la)-\frac{\M!}{\xi^{\M+1}}\right|
}
\\&\stackrel{(\ref{FundShift})}=&
\left|e^{-\xi}\Phi(e^{-\xi},-\M,1)-\frac{\M!}{\xi^{\M+1}}-\sum_{m=1}^{\la-1}e^{-m\xi}m^N\right|
\nonumber
\\&=&\left|\frac{\M!}{(2\pi)^{\M+1}}\left(\zeta(\M+1,1+\frac{\xi}{2\pi i})+\zeta(\M+1,1-\frac{\xi}{2\pi i})\right)-\sum_{m=1}^{\la-1}e^{-m\xi}m^N\right|
\\&\leq&\frac{2\M!\zeta(\M+1)}{(2\pi)^{\M+1}}+\zeta(-N)-\zeta(-N,\la)
\end{eqnarray}
and thus a bound $C(1,-N,\la)=\M!\frac{2\zeta(\M+1)}{(2\pi)^{\M+1}}+\zeta(-N)-\zeta(-N,\la)$ as above.
The term $\frac1u$ provides a first summand equal to
\[
\frac\partial{\partial s}_{|s=-\ms}\frac1{\Gamma(s)}\int_0^\infty\frac1uu^{s-1}\frac{du}{e^{vu}}=v^{\ms+1}\left(\frac{\log v}{\ms+1}-\frac{1}{(\ms+1)^2}\right).
\]
In this case $z=1$ the error term $R$ is given by
\begin{align*}
\lefteqn{R(1,-\M,\la,\ms,v)}\\=&
\frac{(-1)^{N}v^{N-\ms}N!}{(N-\ms-1)!}
\int_0^\infty\left(\frac{e^{-\la u}}{1-e^{-u}}-\frac1u-\sum_{m=0}^{\M-1}\frac{(-u)^m\Phi(1,-m,\la)}{m!}\right)\frac{du}{u^{\ms+1}e^{vu}}.
\qedhere
\end{align*}
\end{proof}

\begin{cor}\label{AsympDim=0}
For $e^{i\vp}\neq1$, $\M>0$, $\frac{\partial\Phi}{\partial s}$ has the asymptotic expansion for $v\to\infty$
\begin{eqnarray*}
\frac{\partial\Phi}{\partial s}(e^{i\vp},0,v+\la)&=&\frac{-\log v}{1-e^{i\vp}}
+\sum_{m=1}^{\M-1}\frac{\Phi(e^{i\vp},-m,\la)}{(-v)^{m}m}
+\frac{R(e^{i\vp},-\M,\la,\ms,v)}{(-v)^{\M}\M}.
\end{eqnarray*}
For $e^{i\vp}=1$, one obtains
\begin{eqnarray*}
\frac{\partial\Phi}{\partial s}(1,0,v+\la)&=&v \log v-v+(\la-\frac{1}{2})\log v
+\sum_{m=1}^{\M-1}\frac{\Phi(1,-m,\la)}{(-v)^{m}m}
+\frac{R(1,-\M,\la,\ms,v)!}{(-v)^{\M}\M}.
\end{eqnarray*}
\end{cor}
\section{Asymptotic expansion of sums over polynomials times logarithms}
In this section, using Th.\ \ref{AsymptLerch}, we derive formulae for the asymptotic expansion of sums of the form $\sum_{k=1}^vP(k)\log k$, 
where
$P:\BZ\to\BC$ is a function of the form
\begin{equation} P(k)=\sum_{j=0}^r c_jk^{\ms_j} e^{ik\phi_j}\label{sternnn}
\end{equation} with $r\in\BN_0$, $\ms_j\in\BN_0$, $c_j\in\BC$,
$\phi_j\in\BR$ for all $j$. We shall use linear operators $\bz,\bzs,P^*$ etc. as in \cite[p.\ 102]{K2}: For $p\in\BR$ we shall use the notation $P^\odd(k):=(P(k)-P(-k))/2$, furthermore $\bz,\bzs$ as introduced in Def. \ref{zetama} and
\begin{align*}
P^*(p)&:=-\sum_{j=0\atop\phi_j\equiv 0{\rm\ mod\,
}2\pi}^r c_j \frac{p^{\ms_j+1}}{4(\ms_j+1)} \sum_{\ell=1}^{\ms_j}\frac 1 \ell,
\qquad\qquad\tilde P^*(p)&:=\sum_{j=0\atop\phi_j\equiv 0{\rm\ mod }2\pi}^r c_j
\frac{p^{\ms_j+1}}{2(\ms_j+1)^2},\\
\mbox{Res\,}P(p)&:=\sum_{j=0\atop\phi_j\equiv 0{\rm\ mod }2\pi}^r c_j
\frac{p^{\ms_j+1}}{2(\ms_j+1)}.&
\end{align*}
The variable on which these operators act will always be denoted by $k$ in this article.
\begin{Bem}\label{informal}
Informally, $\bz$ and $\bzs$ can be interpreted as $"\bz P=\sum_{k=1}^\infty P(k)"$ and $"\bzs P=-\sum_{k=1}^\infty P(k)\log k"$, where the right hand sides are not well-defined.
\end{Bem}
Notice that
\begin{equation}\label{bzSymmetry}
\bz(P(k)+P(-k))=-P(0)
\end{equation}
by \cite[1.11(17)]{Erd1} or \cite[Eqs.\ (61), (63)]{K2} (or deduce it from Eq.\ (\ref{IterationPhi})). 
Similarly, differentiating Jonqui\`ere's relation \cite[1.11(16)]{Erd1} by $s$ shows for $\vp\in[0,2\pi]$, $n\in\BZ^+$ that
\begin{eqnarray}\label{bzsSymmetry}\nonumber
\bzs (k^ne^{ik\vp}+(-k)^ne^{-ik\vp})&=&e^{i\vp}\frac\partial{\partial s}\Phi(e^{i\vp},-n,1)+(-1)^ne^{-i\vp}\frac\partial{\partial s}\Phi(e^{-i\vp},-n,1)
\\&=&\pi ie^{i\vp}\Phi(e^{i\vp},-n,1)+\frac{\ms!}{(-2\pi i)^n}\Phi(1,n+1,\frac{\vp}{2\pi}).
\end{eqnarray}
Define $R^{\rm rot}(\vp):={\rm Im\,}\big(e^{i\vp}\frac\partial{\partial s}\Phi(e^{i\vp},0,1)\big)$ as in \cite[p.\ 559]{K1}, where this special function is investigated further. Using the Digamma function $\psi$, \cite[Lemma 13]{KReid} shows that for $0<\vp<2\pi$
\begin{equation}\label{Rrot-Formel}
\bzs e^{ik\vp}=e^{i\vp}\frac{\partial\Phi}{\partial s}(e^{i\vp},0,1)=\frac{-\log 2\pi+\Gamma'(1)}2-\frac{\psi(\frac{\vp}{2\pi})+\psi(1-\frac{\vp}{2\pi})}4+iR^{\rm rot}(\vp).
\end{equation}
\begin{defin}
For $a\in\BN_0$ define the linear operator $\bz^R_{a,v,\M}$ on functions of the form
\begin{equation}\label{}
(k,v)\mapsto \sum_{j=0}^rc_jv^{r_j}k^{\ms_j}e^{ik\vp_j+iv\psi_j}
\end{equation}
on $(\BR^+)^2$ with $m,\ms_j\in\BN_0$, $r_j\in\BZ$, $\ms_j+r_j+\M\geq0$, $c_j\in\BC$, $\phi_j,\psi_j\in\BR$ for all $j$, such that
\begin{eqnarray*}
\bz^R_{a,v,N}(v^rk^ne^{ik\vp+iv\psi}):=v^{-N}e^{-i(a+v+1)\vp+iv\psi}\frac{(-1)^{\ms+r+\M}}{\left({\ms+r+\M\atop n}\right)(\M+r)}R(e^{-i\vp},-(n+r+\M),\la+1,n,v).
\end{eqnarray*}
Also denote by $f\mapsto f^{[\deg_v>-\M]}$, $f\mapsto f^{[\deg_v=-\M]}$ the linear map which maps formal power series of the form $e^{i(v+b)\vp}\sum_{k=-\infty}^{c}a_kv^k$ to $e^{i(v+b)\vp}\sum_{k=1-\M}^{c}a_kv^k$, $e^{i(v+b)\vp}a_{-\M} v^{-\M}$ resp. Let $\log^\ddagger(x+1):=-\sum_{k=1}^{\infty}\frac{(-x)^k}k$ be a formal power series.
\end{defin}
This last term will only appear in formulae where only a finite partial sum of $\log^\ddagger$ is used. \begin{prop}\label{estimate2}
Consider $\vp\in]-\pi,\pi]$, $\M\in\BZ^+$ and $a\in\BN_0$. When setting $\vp_0:=\left\{{2\pi\atop\vp}\right.{{\rm if\,}\vp=0,\atop{\rm else}}$ one finds
\begin{eqnarray*}
\frac{C(z,-N,a)}{\M!}\leq\frac{2\zeta(2)}{|\vp_0|^{N+1}}+\frac{a^{\M+1}}{(\M+1)!}
\end{eqnarray*}
and thus
\begin{eqnarray*}
|\bz^R_{a,v,N}(v^rk^ne^{ik\vp+iv\psi})|&\leq& v^{-N}n!(r+\M-1)!\left(\frac{2\zeta(2)}{|\vp_0|^{n+r+N+1}}+\frac{a^{n+r+\M+1}}{(n+r+\M)!}\right)
.
\end{eqnarray*}
\end{prop}
\begin{proof}
In the case $0<\vp\leq \pi$ the term $\zeta(N+1,\frac{\vp}{2\pi})+\zeta(N+1,1-\frac{\vp}{2\pi})-(\frac{\vp}{2\pi})^{-\M-1}=\zeta(N+1,1+\frac{\vp}{2\pi})+\zeta(N+1,1-\frac{\vp}{2\pi})$ has a global maximum at $\vp=\pi$, given by $\zeta(N+1,\frac32)+\zeta(N+1,\frac12)=2(2^{\M+1}-1)\zeta(\M+1)-2^{\M+1}<2^{\M+1}\cdot(2\zeta(2)-1)$. Furthermore $\zeta(-N)-\zeta(-N,\la)=\sum_{m=1}^{a-1}m^N<|a-1|\cdot|a-1|^\M=|a-1|^{\M+1}$.
Thus
\begin{align*}
\frac{C(z,-N,a)}{\M!}&<\vp^{-N-1}\cdot\bigg(1+(\frac\vp\pi)^{\M+1}(2\zeta(2)-1)\bigg)+\frac{|a-1|^{\M+1}}{\M!}
\leq\frac{2\zeta(2)}{\vp^{N+1}}+\frac{|a-1|^{\M+1}}{\M!}.
\end{align*}
For $e^{i\vp}=1$ one finds
\[
\frac{C(z,-N,a)}{\M!}\leq\frac{2\zeta(2)}{(2\pi)^{\M+1}}+\frac{|a-1|^{N+1}}{\M!}.
\qedhere
\]
\end{proof}
In the case $\vp\not\equiv2\pi$, $\bz^R_{a,v,N}(v^rk^ne^{ik\vp+iv\psi})$ equals
\begin{align*}
v^{r}e^{iv\psi-i(a+v+1)\vp}n!
\int_0^\infty\left(\frac{e^{-\la u}}{e^u-e^{-i\vp}}-\sum_{j=0}^{\ms+r+\M-1}\frac{(-u)^j\Phi(e^{-i\vp},-j,\la+1)}{j!}\right)\frac{du}{u^{\ms+1}e^{vu}}.
\end{align*}
\begin{prop}\label{PHatStern}
Assume that the function $P$ verifies the symmetry $P(k)=-P(v+a-k)$.
Then $\mbox{\rm Res}\,P(v+a)=0$ and $\tilde P^*(v+a)=P^*(v+a)$ hold.
\end{prop}
\begin{proof}
For any polynomial $P(k)$ satisfying the above symmetry one obtains
\[
\mbox{Res}\,P(v+a)=\frac12\int_0^{v+a}P(k)\,dk=0
\]
(compare \cite[Eq.\ (60)]{K2}). For $P(k)=k^{\ms}$ one finds
\[
\tilde P^*(p)-\mbox{Res}\,P(p)\cdot\log p=\frac{p^{\ms+1}}2\left(\frac1{(n+1)^2}-\frac{\log p}{\ms+1}\right)=-\frac12\int_0^p k^{\ms}\log k\,dk.
\]
Furthermore \cite[4.253.1]{GrRy} shows for $P(k):=\frac{k^{\ms}-(v+a-k)^{\ms}}2$ that
\begin{align*}
\tilde P^*(v+a)-\mbox{Res}\,P(v+a)&=-\frac12\int_0^{v+a}\frac{k^{\ms}-(v+a-k)^{\ms}}2\log k\,dk
\\&=-(v+a)^{\ms+1}\frac{{\cal H}_{\ms}}{4(\ms+1)}=(k^{\ms})^*(v+a).
\end{align*}
Hence $\tilde P^*(v+a)=P^*(v+a)$ if the assumed symmetry holds.
\end{proof}
\begin{prop}\label{AsympMitZeta}
Consider $\M\in\BZ^+$, $\la\in\BN_0$. Let $P:\BZ\to\BC$ be a function of the form
$k\mapsto\sum_{j=0}^m c_jk^{m_j} v^{\ms_j}e^{ik\phi_j}$. Then there is the aymptotic expansion for $v\to\infty$
\begin{align*}
\sum_{k=1}^{v+a}P(k)\log k=&
-2\tilde P^*(v+a)+2\mbox{\rm Res}P(v+a)\cdot\log v
\\&-\log v\cdot\bz\left(P(k+a+v)\right)-\bzs\left(P(k)\right)
\\&+\left(2\mbox{\rm Res}P(v+a)\cdot\log^\ddagger\left(\frac{a}{v}+1\right)\right)^{[\deg_v>-\M]}
\\&-
\bz\left(P(k+a+v)\log^\ddagger\left(\frac{k+a}{v}+1\right)\right)^{[\deg_v>-\M]}
\\&+\bz^R_{a,v,\M}(P(-k))
\end{align*}
where the error-term is $O(v^{-\M})$.
\end{prop}
Because of the influence of the $v^{n_j}$-factors the error term cannot easily be described as an operator acting on $P(v+a)\cdot\log^\ddagger\left(\frac{a}{v}+1\right)$. When one does this, either the degrees $n_j$ or $m_j$ would have to be incorporated in such an operator.
\begin{proof}
Because of the linearity it is sufficient to consider $P(k):=k^{\ms}v^re^{ik\vp}$. All of the operators except $\bz^R$ are linear with respect to the factor $v^r$.
In the case $e^{i\vp}\neq1$, the two expressions for multiples of $\log v$ in Th.\,\ref{AsymptLerch} show that
\begin{align}
\label{MultLogv}
\bz\left((k+a)^{\ms}e^{i(k+a)\vp}\right)
=e^{ia\vp}\sum_{j=0}^n\left({\ms\atop j}\right)a^{\ms-j}e^{i\vp}\Phi(e^{i\vp},-j,1)
\stackrel{\rm Th.\,\ref{AsymptLerch}}=e^{i(a+1)\vp}\Phi(e^{i\vp},-n,a+1).
\end{align}
Informally in the sense of Remark \ref{informal} this can be regarded as the equation
$$
''\bz\left((k+a)^{\ms}e^{i(k+a)\vp}\right)=\sum_{k=0}^\infty(k+a)^{\ms}e^{i(k+a)\vp}=e^{i(a+1)\vp}\Phi(e^{i\vp},-n,a+1)''
$$
where the middle term is not well-defined.
In the case $|\frac{a+k}{v}|<1$, deriving the Taylor expansion
$$
(a+k+v)^{-s}=v^{-s}\left(\frac{a+k}{v}+1\right)^{-s}=\sum_{m=0}^\infty v^{-m-s}\left({-s\atop m}\right)\cdot(k+a)^m
$$
shows that
\begin{eqnarray}
\nonumber\lefteqn{
-(a+k+v)^{\ms}\cdot\left[
\left(\log v+\log\left(\frac{a+k}{v}+1\right)\right)
\right]
=\frac\partial{\partial s}_{|s=-n}v^{-s}\left(\frac{a+k}{v}+1\right)^{-s}
}\\
\label{Taylor21}
&=&\sum_{m=0}^\infty (k+a)^m\frac\partial{\partial s}_{|s=-n}v^{-m-s}\left({-s\atop m}\right)
\\
\nonumber
&\stackrel{(\ref{AblBinomial})}=&\sum_{m=0}^{\ms}v^{m}\left({\ms\atop m}\right)\left(-\log v+{\cal H}_{m}-{\cal H}_{\ms}\right)\cdot(k+a)^{\ms-m}
+\sum_{m=1}^{\infty}\frac{(-v)^{-m}\cdot(k+a)^{m+n}}{m\left({m+n\atop n}\right)}
\\
\nonumber
&=&-(a+k+v)^{\ms}\log v+
\sum_{m=0}^{\ms}v^{m}\left({\ms\atop m}\right)\left({\cal H}_{m}-{\cal H}_{\ms}\right)\cdot(k+a)^{\ms-m}
\\&&\nonumber
+\sum_{m=1}^{\infty}\frac{(-v)^{-m}\cdot(k+a)^{\ms+m}}{m\left({m+n\atop n}\right)}.
\end{eqnarray}
Then for $e^{i\vp}\neq1$, $n\in\BN_0$ Th.\ \ref{AsymptLerch} implies
\begin{align}
\nonumber
\sum_{k=1}^{v+a}v^rk^{\ms}e^{ik\vp}\log k=&v^re^{i(v+a+1)\vp}\frac{\partial\Phi}{\partial s}(e^{i\vp},-n,v+a+1)
-v^re^{i\vp}\frac{\partial\Phi}{\partial s}(e^{i\vp},-n,1)
\\
\nonumber
\stackrel{\rm Th.\ \ref{AsymptLerch}}=&e^{i(v+a+1)\vp}\sum_{m=0}^{\M+n+r-1}\Phi(e^{i\vp},-m,\la+1)\frac\partial{\partial s}_{|s=-n}v^{r-m-s}\left({-s\atop m}\right)
\\
\nonumber
&+e^{i(v+a+1)\vp}R(e^{i\vp},-\M-n-r,\la+1,n,v)
\\
\nonumber
&\cdot\frac\partial{\partial s}_{|s=-n}v^{-\M-n-s}\left({-s\atop \M+n+r}\right)
-e^{i\vp}\frac{\partial\Phi}{\partial s}(e^{i\vp},-n,1)v^r
\\
\nonumber
\stackrel{(\ref{MultLogv})}=&\sum_{m=0}^{\M+n+r-1}\bz(e^{i(v+a+k)\vp}(k+a)^m)
\frac\partial{\partial s}_{|s=-n}v^{r-m-s}\left({-s\atop m}\right)
\\
\nonumber
&+e^{i(v+a+1)\vp}\Phi(e^{i\vp},-n,v+a+1)v^r\log v
\\
\nonumber
&+\bz^R_{a,v,\M}(v^r(-k)^{\ms}e^{-ik\vp})
\\
\label{abc}
&-e^{i(v+a+1)\vp}\Phi(e^{i\vp},-n,v+a+1)v^r\log v-e^{i\vp}\frac{\partial\Phi}{\partial s}(e^{i\vp},-n,1)v^r
\\
\nonumber
\stackrel{(\ref{Taylor21})}=&\bz\left(-e^{i(a+k+v)\vp}v^r(k+a+v)^{\ms}\log^\ddagger\left(\frac{k+a}{v}+1\right)\right)^{[\deg_v>-\M]}
\\
\nonumber
&+\bz^R_{a,v,\M}(v^r(-k)^{\ms}e^{-ik\vp})
\\&-\log v\cdot\bz\left(e^{i(a+k+v)\vp}v^r(k+a+v)^{\ms}\right)-\bzs\left(v^rk^{\ms}e^{ik\vp}\right).
\label{SummemitZeta}
\end{align}
In the case $e^{i\vp}=1$, the analogue of Eq.\ (\ref{MultLogv}) takes the form
\begin{align}\label{MultLogv0}
-\frac{a^{m+1}}{m+1}+\bz\left((k+a)^{m}e^{i(k+a)\vp}\right)
\stackrel{\rm Th.\,\ref{AsymptLerch}}=e^{i(a+1)\vp}\Phi(e^{i\vp},-m,a+1).
\end{align}
Therefore in this case there is an additional summand on the right hand side of Eq.\ (\ref{abc})
\begin{eqnarray}\nonumber
\lefteqn{
-\sum_{m=0}^{\M+n+r-1}\frac{a^{m+1}}{m+1}\frac\partial{\partial s}_{|s=-n}v^{r-m-s}\left({-s\atop m}\right)
\stackrel{(\ref{Taylor21})}=\left(\frac\partial{\partial s}_{|s=-n}v^{r-s}\int_0^a\left(\frac{\tilde a}{v}+1\right)^{-s}d\tilde a\right)^{[\deg_v>-\M]}
}\\&
=&(v^r(v+a)^{\ms+1}-v^{r+n+1})\left(\frac{\log v}{\ms+1}-\frac1{(n+1)^2}\right)+
\left(\frac{v^r(v+a)^{\ms+1}\log(\frac{a}v+1)}{\ms+1}\right)^{[\deg_v>-\M]}.
\label{wbouxe}
\end{eqnarray}
Furthermore the asymptotic in Th.\ \ref{AsymptLerch} contains additional terms of order $v^{r+n+1}$ and $v^{r+n+1}\log v$ which are independent of $a$ and cancel with the multiple of $v^{r+n+1}$ in Eq.\ (\ref{wbouxe}). The second one is the top term of $-\log v\cdot\bz\left(e^{i(a+k+v)\vp}v^r(k+a+v)^{\ms}\right)$ and it takes the form $-v^r\log v\cdot 2\mbox{Res}(k^{\ms})(v)$. Thus for arbitrary $\vp\in\BR$ one obtains the same formula as in Eq.\ (\ref{SummemitZeta}) with additional terms, which vanish in the case $e^{i\vp}\neq1$:
\begin{align*}
\sum_{k=1}^{v+a}v^rk^{\ms}e^{ik\vp}\log k=&
-2\tilde P^*(v+a)+2\mbox{Res}P(v+a)\cdot\log v
\\&+\left(2\mbox{Res}P(v+a)\cdot\log^\ddagger\left(\frac{a}{v}+1\right)\right)^{[\deg_v>-\M]}
\\&+
\bz\left(-e^{i(a+k+v)\vp}v^r(k+a+v)^{\ms}\log^\ddagger\left(\frac{k+a}{v}+1\right)\right)^{[\deg_v>-\M]}
\\&+\bz^R_{a,v,\M}(v^r(-k)^{\ms}e^{-ik\vp})
\\&-\log v\cdot\bz\left(e^{i(a+k+v)\vp}v^r(k+a+v)^{\ms}\right)-\bzs\left(v^rk^{\ms}e^{ik\vp}\right).
\qedhere
\end{align*}
\end{proof}

\section{Definition of analytic torsion and certain characteristic classes}
Let $(M,g)$ be a compact Hermitian manifold of complex dimension $n$, acted upon by a holomorphic isometry $t$. Let $\omega^{TM}$ denote the K\"ahler form. Corresponding to the decomposition $TM\otimes\BC=TM^{1,0}\oplus TM^{0,1}$ define $U=:U^{1,0}+U^{0,1}$ for $U\in TM\otimes\BC$. Consider a $t$-equivariant holomorphic vector bundle $E\to M$.
Equip $E$ with an Hermitian metric. Then the vector space $\frak A^{p,q}(M,E)$  of forms of holomorphic degree $p$ and anti holomorphic degree $q$ with coefficients in $E$ shall be equipped with the $L^2$-metric
$$
(\eta,\eta'):=\int_M\<\eta_x,\eta'_x\>\frac{(\omega^{TM})^{\wedge n}}{(2\pi)^nn!}.
$$
Given a differential form $\beta$ we denote its part in degree $q$ by $\beta^{[q]}$.
Let $\bar\partial:\frak A^{p,q}(M,E)\to\frak A^{p,q+1}(M,E)$ denote the Dolbeault operator with adjoint $\bar\partial^*$ and let $\square_q:=(\bar\partial+\bar\partial^*)^2:\frak A^{0,q}(M,E)\to\frak A^{0,q}(M,E)$ be the Kodaira-Laplace operator. The isometry $t$ induces an action $t^*$ on $\frak A^{p,q}(M,E)$. For ${\rm Re\,}s>>0$ the zeta function
$$Z(s):=\sum_{q=0}^{n}(-1)^{q+1}q\Tr (t^*\circ(\square_{q|(\ker\square_q)^\perp})^{-s})$$
is well-defined and it has a holomorphic continuation to $s=0$. Following \cite[Section 1]{K1} we set:
\begin{defin}
The $t$-equivariant complex analytic torsion is defined as
$
T_t(M,\mtr E):=Z'(0).
$
\end{defin}

Consider a holomorphic vector bundle $E$. Let $\nabla^E$ be the associated canonical covariant derivative with curvature $\Omega^E\in\frak A^{1,1}(M,\End E)$. 
Let $t$ be a holomorphic isometry acting on the Hermitian manifold $M$. Assume that $E$ is $t$-invariant as a holomorphic Hermitian bundle and that $E$ is equipped with an equivariant structure $t^E$. The
Hermitian vector bundle $\mtr E$ splits on the fixed point submanifold $M^t$ into a direct sum
$\bigoplus_{\zeta\in S^1}\mtr E_\zeta$, where the equivariant structure $t^E$ of $E$
acts on
$\mtr E_\zeta$ as $\zeta$. Then the Chern character form on $M^t$ is defined as
\begin{eqnarray*}
\ch_t(\mtr E)&:=&\sum_\zeta \zeta\ch(\mtr E_\zeta)\\ &=&\Tr t^E+\sum_\zeta \zeta
c_1(\mtr E_\zeta)+\sum_\zeta \zeta \left(\frac{1}{2}c_1^2(\mtr E_\zeta) -c_2(\mtr
E_\zeta)\right)+\dots
.
\end{eqnarray*}
The Todd form of an equivariant vector bundle is defined as
$$
\Td_t(\mtr E):=\frac{c_{{\rm rk}\,E_g}(\mtr E_g)}
{\ch_t(\sum_{j=0}^{{\rm rk}\,E} (-1)^j \Lambda^j \mtr E^*)}.
$$
Let $G_t$ denote the additive characteristic class given by
$$
G_t(L):=\bzs\ch_t(L^{\otimes k})=\sum_{m=0}^\infty e^{i \vp}\frac{\partial\Phi}{\partial s}(e^{i \vp},-m,1)\frac{c_1(L)^m}{m!}
$$
on a line bundle $L$ as above. Set as in \cite[Def. 1.5]{BG}
$$
\Td'_t(\mtr E):=\frac\partial{\partial \vep}_{|\vep=0}
\Td\Big(\frac{-1}{2\pi i}\Omega^E+\vep\Big)\prod_j \frac{\Td}{c_\top}\Big(\frac{-1}{2\pi i}\Omega^E+i\theta_j+\vep\Big).
$$
Hence $\frac{\Td'_t}{\Td_t}$ is the additive characteristic class given by
$$
\frac{\Td'_t}{\Td_t}(L)=\left\{
{\frac1{1-e^{c_1(L)}}+\frac1{c_1(L)}=1-\frac1{1-e^{-c_1(L)}}+\frac1{c_1(L)}
\atop
\frac1{1-e^{c_1(L)+i\vp}}=1-\frac1{1-e^{-c_1(L)-i\vp}}
}\right.\mbox{ if }
{\vp=0,\atop\vp\neq0}
$$
for a line bundle with $t$-action given by $e^{i\vp}$ on a component of the fixed point set. By the Taylor series (\ref{nonweirdTaylor2}) and (\ref{weirdTaylor2}) we see that
\begin{equation}\label{zetaTd'}
\frac{\Td'_t}{\Td_t}(L)=1+\bz\ch_t(L^{\otimes k})
=1+\sum_{j=0}^\infty\frac{e^{i\vp}\Phi(e^{i\vp},-j,1)}{j!}c_1(L)^j
\stackrel{(\ref{FundShift})}=\sum_{j=0}^\infty\frac{\Phi(e^{i\vp},-j,0)}{j!}c_1(L)^j.
\end{equation}
In particular in the case of isolated fixed points one finds
$$
\Td'_g(\mtr{TM})=\prod_j \frac{1}{1-e^{-i\theta_j}}\cdot\sum_j\frac{1}{1-e^{i\theta_j}}.
$$
Set $\Td^*_g(TM):=\prod_{j\atop \theta_j\notin 2\pi \BZ} \frac{1}{1-e^{-i\theta_j}}\cdot\sum_j
\left\{
{1/2\mbox { for }\theta_j\in 2\pi \BZ,\atop\frac{1}{1-e^{-i\theta_j}}\mbox{ else.}}
\right.$

\section{Asymptotic expansion of the torsion}\label{SectionAsymptSymm}
We use for symmetric and homogeneous spaces the same conventions as in \cite{KK}.
Let $G$ be a compact Lie group with neutral element $e_G$.  Fix a choice of a maximal torus $T\subset G$ and an ordering $\Sigma_G=\Sigma^+\cup\Sigma^-$. Let  $W_G$ denote the associated Weyl group, $\frak g$ the Lie algebra, $\Sigma_G$ the roots and $\rho_G:=\frac12\sum_{\a\in\Sigma^+}\a$ the Weyl vector. We denote the $G$-representation with highest weight $\lambda$ by $V^G_{\rho_G+\lambda}$ and its character by $\chi^G_{\rho+\lambda}$. These notations are extended to virtual representations in the representation ring to arbitrary weights $\lambda$ via the action of $W_G$ on $\rho+\lambda$. 
Fix $X_0\in\frak t$ in the closure of the positive Weyl chamber and let $K$ denote its stabilizer with respect to the ${\rm Ad}_G$-action. Set $\Psi:=\Sigma_G\setminus\Sigma_K$. Let $g_{X_0}$ denote the induced K\"ahler metric on $G/K$, let $\<\cdot,\cdot\>$  denote the associated ${\rm Ad}_G$-invariant scalar product on $\frak t^*$ and set $\a^\vee:=\frac{2\a}{\|\a\|^2}$ and $\Psi^+:=\{\a\in\Psi|\<\a^\vee,\rho+\lambda\>\geq0\}$, $\Psi^-:=\{\a\in\Psi|\<\a^\vee,\rho+\lambda\><0\}$. Every complex $K$-representation $V^K_{\rho_K+\mu}$ induces a $G$-invariant holomorphic vector bundle $E_{\rho+\mu}$ on $G/K$. If $\dim V^K_{\rho_K+\mu}=1$, we shall often denote this vector bundle by $L_{\rho+\mu}$.

For the sake of brevity, we shall often drop the index $G$ for objects associated to the Lie group specifically denoted as $G$ and simply use the notations $\rho,V_{\rho+\lambda},\chi_{\rho+\lambda}$.

We shall need the following two results about representations $V^K_{\rho_K+\lambda}$ of the abelian summand of the Lie algebra of a compact Lie group $K$ in the proof of Theorem \ref{AsymptSymmetricSpaces}.
\begin{prop}\label{TensorZerlegung}
Let $K$ be a compact Lie group. We choose a fixed maximal torus $T\subset K$ and a set of positive roots $\Sigma^+_K$. If $\lambda$ is a weight such that $\forall \alpha\in\Sigma^+_K:\<\alpha^\vee,\lambda\>=0$, then for any arbitrary weight $\mu$ and $\ell\in\BZ$ one finds
$$
V^K_{\rho_K+\ell\lambda+\mu}=(V^K_{\rho_K+\lambda})^{\otimes\ell}\otimes V^K_{\rho_K+\mu}.
$$
\end{prop}
The weights $\lambda$ and $\mu$ do not need to be dominant.
\begin{proof}
The group $W_K$ is generated by the reflections $s_\alpha(\beta)=\beta-\<\alpha^\vee,\beta\>\cdot\alpha$ in the roots. Thus $\lambda$ is invariant under the action of $W_K$. The Weyl character formula shows that on generic elements of $K$ one gets the identity of characters
\[
\chi^K_{\rho_K+\ell\lambda+\mu}=\frac{\sum_{w\in W_K}(-1)^{{\rm sign}\, w}e^{2\pi i w(\rho_K+\ell\lambda+\mu)}}{\prod_{\alpha\in\Sigma^+_K}(e^{\pi i w\alpha}-e^{-\pi i w\alpha})}
=e^{2\pi i\ell\lambda}\frac{\sum_{w\in W_K}(-1)^{{\rm sign}\, w}e^{2\pi i w(\rho_K+\mu)}}{\prod_{\alpha\in\Sigma^+_K}(e^{\pi i w\alpha}-e^{-\pi i w\alpha})}.
\qedhere
\]
\end{proof}
\begin{Bem}A more functorial way of looking at this lemma is by means of the same argument as in \cite[proof of Lemma 12]{K2}, employing Adams operators $\psi^\ell$: Using the notation there, $\tilde\pi^!V^K_{\rho_K+\lambda}=V^T_{\lambda}$ holds as both sides are 1-dimensional.
Hence one obtains $V^K_{\rho_K+\ell\lambda+\mu}=\tilde\pi_!(\psi^\ell V^T_{\lambda}\otimes V^T_{\mu})=\tilde\pi_!(\psi^\ell\tilde\pi^!V^K_{\rho_K+\lambda}\otimes V^T_{\mu})=\psi^\ell V^K_{\rho_K+\lambda}\otimes V^K_{\rho_K+\mu}$.
\end{Bem}
The following result generalises \cite[p.\ 663, Remark 1.]{KK}, where it has been proven in the case $t=e_G$:
\begin{prop}\label{AbschDegelldurchDim}
Consider the same setting as in Proposition  \ref{TensorZerlegung}. For any weight $\mu$ and any $t\in G$ one finds that
$$
\deg_\ell\chi_{\rho+\ell\lambda+\mu}(t)\leq\dim_\BC (G/K)^t.
$$
\end{prop}
\begin{proof}
By Proposition \ref{TensorZerlegung} and the Atiyah--Segal--Singer fixed point formula as in \cite[Th. 13]{K2}, we get
\begin{eqnarray}
\chi_{\rho+\ell\lambda+\mu}(t)&=&\int_{M^t}\Td_t(TM)\ch_t(E_{\rho_K+\ell\lambda+\mu})
\nonumber
\\&=&\int_{M^t}\Td_t(TM)\ch_t(L_{\rho_K+\lambda})^\ell\ch_t(E_{\rho_K+\mu})
\label{MtfuerCharakter}
\\&=&\int_{M^t}\Td_t(TM)(\ch_t(L_{\rho_K+\lambda})^{[0]})^\ell e^{\ell c_1(L_{\rho_K+\lambda})}\ch_t(E_{\rho_K+\mu}).
\nonumber
\end{eqnarray}
The right hand side has degree in $\ell$ less or equal to $\dim_\BC (G/K)_t$.
\end{proof}
For a general Hermitian symmetric space, the value of the torsion is given by \cite[Th. 5.2]{KK} in terms of characters $\chi_{\rho+\ell \lambda+k\alpha}$. Notice that for a fixed $t\in G$, $\chi_{\rho+\ell\lambda+k\alpha}(t)$ is a linear combination of terms of the form $k^{m}\ell^{m'}e^{i\vp'\ell+i\vp k}$ with $\vp,\vp'\in\BR$, $m,m'\in\BN_0$ by \cite[Eq.\ (66)]{K2}. As explained in \cite[Eq.\ (66)]{K2}, for $t=e^X$ the angles $\vp$, $\vp'$ are always of the form $2\pi w\a(X)$, $2\pi w'\lambda(X)$ with $w,w'\in W_G$.
The operators $\bz,\bzs$ do not change the growth in $\ell$.

\begin{theor}\label{AsymptSymmetricSpaces}
Let $M=G/K$ be a compact Hermitian symmetric space and let $\mtr L:=\mtr L_{\rho_K+\lambda}$ be a $G$-invariant holomorphic Hermitian line bundle on $M$. Assume that $\mtr L$ is positive in the sense that $\forall\alpha\in\Psi:\<\alpha^\vee,\lambda\>>0$. Fix a K\"ahler metric $g_{X_0}$ on $M$. Choose $X\in{\frak t}$ and set $t:=e^X$. Then the asymptotic behaviour of the equivariant analytic torsion for high powers of $\mtr L$ is given by 
\begin{eqnarray}\nonumber
\lefteqn{
T_t(G/K,\mtr L^\ell_{\rho_K+\ell\lambda})=
-\sum_{\a\in\Psi}\log\frac{\langle\a^\vee,\ell\lambda\rangle}{\a^\vee(X_0)}\cdot\bz\chi_{\rho+\ell\lambda+k\alpha}(t)
+\sum_{\a\in\Psi} \bzs
\chi_{\rho+\ell\lambda+k\a}(t)
}
\\&&\nonumber
-\sum_{\a\in\Psi}\bz\left(\chi_{\rho+\ell\lambda+k\alpha}(t)
\cdot\log^\ddagger\left(\frac{\langle\a^\vee,\rho\rangle+k}{\langle\a^\vee,\ell\lambda\rangle}+1\right)\right)^{[\deg_\ell>-\M]}
\\&&-\sum_{\a\in\Psi}\bz^R_{\langle\a^\vee,\rho\rangle,\langle\a^\vee,\lambda\rangle\ell,\M}\chi_{\rho+\ell\lambda+k\alpha}(t)
.
\end{eqnarray}
\end{theor}
\begin{proof}
According to \cite[Th. 5.2]{KK}, the equivariant analytic torsion of
$\mtr L^\ell_{\rho_K+\ell\lambda}$ is given by
\begin{eqnarray}\nonumber\label{TorsionAusKK}
\lefteqn{
T_t(G/K,\mtr L^\ell_{\rho_K+\ell\lambda})=-2\sum_{\a\in\Psi} \bzs
\chi^\odd_{\rho+\ell\lambda-k\a}(t)
-2\sum_{\a\in\Psi} 
\chi^*_{\rho+\ell\lambda-k\a}(t)(\langle\a^\vee,\rho+\ell\lambda\rangle)}\\
&&\nonumber
-\sum_{\a\in\Psi} \bz
\chi_{\rho+\ell\lambda-k\a}(t)\cdot\log\a^\vee(X_0)
-\chi_{\rho+\ell\lambda}(t)\sum_{\a\in\Psi^+}\log\a^\vee(X_0)
\\&&
-\sum_{\a\in\Psi}\sum_{k=1}^{\langle\a^\vee,\rho+\ell\lambda\rangle}
\chi_{\rho+\ell\lambda-k\a}(t)\cdot\log k.
\end{eqnarray}
By Eq.\ (\ref{bzSymmetry}),
\begin{eqnarray}\label{zetaSchluckt1}
-\bz\chi_{\rho+\ell\lambda-k\alpha}-\chi_{\rho+\ell\lambda}
=\bz\chi_{\rho+\ell\lambda+k\alpha}.
\end{eqnarray}
In this case of an ample line bundle one gets $\Psi^+=\Psi$. Therefore the $\log\a^\vee(X_0)$-factor simplifies as in \cite[Th. 18]{K2}. We shall use the symmetry
\begin{equation}\label{Lie-symmetry}
\chi_{\rho+\ell\lambda-k\a}=-\chi_{\rho+\ell\lambda+(k-\<\alpha^\vee,\rho+\ell\lambda\>)\a}
\end{equation} 
induced by reflecting at the hyperplane perpendicular to $\a$. 
By Proposition \ref{AsympMitZeta} one obtains
\begin{align}
\label{AsymptSymmetricSpacesProof}
\nonumber
\lefteqn{
-\sum_{\a\in\Psi}\sum_{k=1}^{\langle\a^\vee,\rho+\ell\lambda\rangle}
\chi_{\rho+\ell\lambda-k\alpha}(t)\cdot\log k
}
\\
\nonumber
=&\sum_{\a\in\Psi}\Bigg(2\widetilde{\chi_{\rho+\ell\lambda-k\alpha}}(t)^*(\langle\alpha^\vee,\rho+\ell\lambda\rangle)+\log\langle\a^\vee,\ell\lambda\rangle\cdot
\bz\chi_{\rho+\ell\lambda-(k+\<\alpha^\vee,\rho+\ell\lambda\>)\alpha}(t)
\\
\nonumber
&-2\mbox{Res}\chi_{\rho+\ell\lambda-k\alpha}(t)(\langle\alpha^\vee,\rho+\ell\lambda\rangle)\cdot\log
\langle\a^\vee,\ell\lambda\rangle
\\
\nonumber
&-\left(2\mbox{Res}\chi_{\rho+\ell\lambda-k\alpha}(t)(\langle\alpha^\vee,\rho+\ell\lambda\rangle)
\cdot\log^\ddagger\left(\frac{\langle\a^\vee,\rho\rangle}{\langle\a^\vee,\ell\lambda\rangle}+1\right)\right)^{[\deg_\ell>-\M]}
\\
\nonumber
&+\bz\left(\chi_{\rho+\ell\lambda-(k+\<\alpha^\vee,\rho+\ell\lambda\>)\alpha}(t)
\cdot\log^\ddagger\left(\frac{\langle\a^\vee,\rho\rangle+k}{\langle\a^\vee,\ell\lambda\rangle}+1\right)\right)^{[\deg_\ell>-\M]}
\\
\nonumber
&-\bz^R_{\langle\a^\vee,\rho\rangle,\langle\a^\vee,\ell\lambda\rangle,\M}\chi_{\rho+\ell\lambda+k\alpha}(t)+\bzs\chi_{\rho+\ell\lambda-k\alpha}(t)
\Bigg)
\\
\nonumber
=&\sum_{\a\in\Psi}\Bigg(2\widetilde{\chi_{\rho+\ell\lambda-k\alpha}}(t)^*(\langle\alpha^\vee,\rho+\ell\lambda\rangle)-\log\langle\a^\vee,\ell\lambda\rangle\cdot\bz\chi_{\rho+\ell\lambda+k\alpha}(t)
\\
\nonumber
&-2\mbox{Res}\chi_{\rho+\ell\lambda-k\alpha}(t)(\langle\alpha^\vee,\rho+\ell\lambda\rangle)\cdot\log\langle\a^\vee,\ell\lambda\rangle
\\
\nonumber
&-\left(2\mbox{Res}\chi_{\rho+\ell\lambda-k\alpha}(t)(\langle\alpha^\vee,\rho+\ell\lambda\rangle)
\cdot\log^\ddagger\left(\frac{\langle\a^\vee,\rho\rangle}{\langle\a^\vee,\ell\lambda\rangle}+1\right)\right)^{[\deg_\ell>-\M]}
\\
\nonumber
&-\bz\left(\chi_{\rho+\ell\lambda+k\alpha}(t)
\cdot\log^\ddagger\left(\frac{\langle\a^\vee,\rho\rangle+k}{\langle\a^\vee,\ell\lambda\rangle}+1\right)\right)^{[\deg_\ell>-\M]}
\\&-\bz^R_{\langle\a^\vee,\rho\rangle,\langle\a^\vee,\ell\lambda\rangle,\M}\chi_{\rho+\ell\lambda+k\alpha}(t)
+\bzs\chi_{\rho+\ell\lambda-k\alpha}(t)
\Bigg)
.
\end{align}

By Proposition \ref{PHatStern} the symmetry (\ref{Lie-symmetry}) implies that $\mbox{Res}\chi_{\rho+\ell\lambda-k\alpha}(t)(\langle\alpha^\vee,\rho+\ell\lambda\rangle)=0$ and that the term $\sum_{\a\in\Psi}2\widetilde{\chi_{\rho+\ell\lambda-k\alpha}}(t)^*(\langle\alpha^\vee,\rho+\ell\lambda\rangle)$ cancels with the term $-2\sum_{\a\in\Psi} 
\chi_{\rho+\ell\lambda-k\a}(t)^*(\langle\a^\vee,\rho+\ell\lambda\rangle)$ in Eq.\ (\ref{TorsionAusKK}).
Together one thus obtains the asymptotic expansion
\begin{align*}\nonumber
\lefteqn{
T_t(G/K,\mtr L^\ell_{\rho_K+\ell\lambda})=-2\sum_{\a\in\Psi} \bzs
\chi^\odd_{\rho+\ell\lambda-k\a}(t)
+\sum_{\a\in\Psi} \bz
\chi_{\rho+\ell\lambda+k\a}(t)\cdot\log\a^\vee(X_0)
}
\\&\nonumber
-\sum_{\a\in\Psi}\Bigg(\log\langle\a^\vee,\ell\lambda\rangle\cdot\bz\chi_{\rho+\ell\lambda+k\alpha}(t)
+\bz\left(\chi_{\rho+\ell\lambda+k\alpha}(t)
\cdot\log^\ddagger\left(\frac{\langle\a^\vee,\rho\rangle+k}{\langle\a^\vee,\ell\lambda\rangle}+1\right)\right)^{[\deg_\ell>-\M]}
\\&+\bz^R_{\langle\a^\vee,\rho\rangle,\langle\a^\vee,\ell\lambda\rangle,\M}\chi_{\rho+\ell\lambda+k\alpha}(t)
-\bzs\chi_{\rho+\ell\lambda-k\alpha}(t)
\Bigg)
.
\qedhere
\end{align*}
\end{proof}

\begin{theor}\label{ErsteTerme}
Consider the same situation as in Theorem \ref{AsymptSymmetricSpaces} and the set of values $\{x_j\,|\,j\in J\}:=\{\frac{\langle\a^\vee,\lambda\rangle}{\a^\vee(X_0)}\,|\,\a\in\Psi\}$. As $X_0$, $\lambda$ and the set $\Psi$ are $W_K$-invariant, each $\Psi_j:=\{\a\in\Psi\,|\,\frac{\langle\a^\vee,\lambda\rangle}{\a^\vee(X_0)}=x_j\}$ defines a virtual $K$-representation with weights $\Psi_j$.
Let $E_j\to G/K$ denote the associated vector bundles. Let $e^{i\vp}$ denote the action of $t$ on $L_{\rho_K+\lambda}$ on a component of $M^t$.
The top term of the asymptotic expansion in Th.\ \ref{AsymptSymmetricSpaces} is given by
\begin{eqnarray*}
\lefteqn{
T_t((M,g_{X_0}),\mtr L_{\rho_K+\lambda}^\ell)
}\\&=&-\log\ell\cdot\int_{M^t}(\Td_t'(TM)-n \Td_t(TM))\ch_t(L_{\rho_K+\lambda})^\ell
\\&&-\sum_{j\in J}\log x_j
\int_{M^t}\Td_t(TM)\left(\frac{\Td_t'(E_j)}{\Td_t(E_j)}-\rk E_j\right)\ch_t(L_{\rho_K+\lambda})^\ell
\\&&+\int_{M^t}\Td_t(TM)G_t(TM)\ch_t(L_{\rho_K+\lambda})^\ell
\\&&-\sum_{m=1-\M}^{\dim M^t-1}\ell^m\int_{M^t}\Td_t(TM)\ch_t(L_{\rho_K+\lambda}^\ell)^{[0]}\bz\Bigg(\sum_{\a\in\Psi}\ch_t(E_{\rho_K+k\alpha})
\\&&\cdot
\sum_{j=\max\{1,-m\}}^{\dim M^t-m}\frac{(-1)^{j+1}}{j(m+j)!}\left(\frac{\langle\a^\vee,\rho\rangle+k}{\langle\a^\vee,\lambda\rangle}\right)^j
c_1(L_{\rho_K+\lambda})^{m+j}
\Bigg)
+O(\ell^{-\M})
\\&=&\int_{M^t}\Td_t^*(TM)\ch(L^\ell_{\rho_K+\lambda})\log\ell
+O(\ell^{\dim M^t})
.
\end{eqnarray*}
\end{theor}
\begin{proof}

According to Theorem \ref{AsymptSymmetricSpaces},
\begin{eqnarray*}
T_t(G/K,\mtr L^\ell_{\rho_K+\ell\lambda})&=&
-\log\ell\cdot\sum_{\alpha\in\Psi}\bz\chi_{\rho+\ell\lambda+k\alpha}(t)+O(\ell^{\dim M^t})
\end{eqnarray*}
Using the Adams operator $\psi^k$, \cite[(103),(106)]{K2} shows that
\begin{eqnarray*}
-\log\ell\cdot\bz\sum_{\alpha\in\Psi} \chi_{\rho+\ell\lambda+k\alpha}(t)
&=&-\log\ell\cdot\int_{M^t}
\Td_t(TM)\bz\ch_t(\psi^kTM)\ch_t(L_{\rho_K+\lambda}^\ell)
\\&\stackrel{(\ref{zetaTd'})}=&-\log\ell\cdot\int_{M^t}
(\Td'_t(TM)-\dim M\cdot \Td_t(TM))\ch_t(L_{\rho_K+\lambda}^\ell).
\end{eqnarray*}
Similarly
$$
\sum_{\a\in\Psi} \bzs\chi_{\rho+\ell\lambda+k\a}(t)=\int_{M^t}
\Td_t(TM)G_t(TM)\ch_t(L_{\rho_K+\lambda}^\ell).
$$
By Equation (\ref{MtfuerCharakter}) one finds
\begin{eqnarray*}
\lefteqn{
-\sum_{\a\in\Psi}\log\frac{\langle\a^\vee,\lambda\rangle}{\a^\vee(X_0)}\cdot\bz\chi_{\rho+\ell\lambda+k\alpha}(t)
}
\\&=&-\sum_{\a\in\Psi}\log\frac{\langle\a^\vee,\lambda\rangle}{\a^\vee(X_0)}\cdot
\int_{M^t}\Td_t(TM)\ch_t(L_{\rho_K+\lambda})^\ell\bz\ch_t(E_{\rho_K+k\alpha})
\end{eqnarray*}
By replacing ${\rm Ad}^{1,0}_{G/K}$ in the proof of \cite[Lemma 2]{K2} by $E_j$, one verifies that $\bigoplus_{\a\in\Psi_j}E_{\rho_K+k\alpha}=\psi^kE_j$ and thus Equation (\ref{zetaTd'}) shows that
\begin{eqnarray*}
\lefteqn{
-\sum_{\a\in\Psi}\log\frac{\langle\a^\vee,\lambda\rangle}{\a^\vee(X_0)}\cdot\bz\chi_{\rho+\ell\lambda+k\alpha}(t)
}\\&=&-\sum_{j\in J}\log x_j\cdot
\int_{M^t}\Td_t(TM)\ch_t(L_{\rho_K+\lambda})^\ell\bz\ch_t(\psi^kE_j)
\\&=&-\sum_{j\in J}\log x_j\cdot
\int_{M^t}\Td_t(TM)\left(\frac{\Td_t'(E_j)}{\Td_t(E_j)}-\rk E_j\right)\ch_t(L_{\rho_K+\lambda})^\ell
\end{eqnarray*}
Furthermore
\begin{eqnarray*}
\lefteqn{
-\bz\left(\chi_{\rho+\ell\lambda+k\alpha}(t)
\cdot\log^\ddagger\left(\frac{\langle\a^\vee,\rho\rangle+k}{\langle\a^\vee,\ell\lambda\rangle}+1\right)\right)^{[\deg_\ell=m]}
=-\int_{M^t}\Td_t(TM)\bz\Bigg(\ch_t(E_{\rho_K+k\alpha})
}\\&&\cdot\sum_{j=\max\{1,-m\}}^{\dim M^t-m}\frac{(-1)^{j+1}}j\left(\frac{\langle\a^\vee,\rho\rangle+k}{\langle\a^\vee,\ell\lambda\rangle}\right)^j
\ch_t(L_{\rho_K+\lambda})^{[0]}\frac{\ell^{m+j}c_1(L_{\rho_K+\lambda})^{m+j}}{(m+j)!}
\Bigg)
\\&=&-\ell^m\int_{M^t}\Td_t(TM)\ch_t(L_{\rho_K+\lambda}^\ell)^{[0]}\bz\Bigg(\ch_t(E_{\rho_K+k\alpha})
\\&&\cdot\sum_{j=\max\{1,-m\}}^{\dim M^t-m}\frac{(-1)^{j+1}}{j(m+j)!}\left(\frac{\langle\a^\vee,\rho\rangle+k}{\langle\a^\vee,\lambda\rangle}\right)^j
c_1(L_{\rho_K+\lambda})^{m+j}
\Bigg)
\\&\stackrel{(\ref{Lie-symmetry})}=&\ell^m\int_{M^t}\Td_t(TM)\ch_t(L_{\rho_K+\lambda})^{[0]}\bz\Bigg(\ch_t(E_{\rho_K-(k+\<\alpha^\vee,\rho\>)\alpha})
\\&&\cdot\sum_{j=\max\{1,-m\}}^{\dim M^t-m}\frac{(-1)^{j+1}}{j(m+j)!}\left(\frac{\langle\a^\vee,\rho\rangle+k}{\langle\a^\vee,\lambda\rangle}\right)^j
c_1(L_{\rho_K+\lambda})^{m+j}
\Bigg)
\end{eqnarray*}
where only components of the fixed point submanifold with $\dim M^t>m$ have a nontrivial contribution. Here $\ch_t(L_{\rho_K+\lambda})^{[0]}$ is equal to the action of $t$ on $L_{\rho_K+\lambda}$ on each component of $M^t$.
\end{proof}
Prop.\ \ref{estimate2} implies the following:
\begin{lemma}\label{estimate3}
For $\a\in\Psi$ consider the character $\chi_{\rho+\ell\lambda+k\alpha}(t)=\sum_{m,r\geq0\atop\vp,\psi\in]-\pi,\pi]}f_{\a,m,r,\vp,\psi}e^{ik\vp+i\ell\psi}\ell^{r}k^m$ and coefficients $f_{\a,r,m,\vp,\psi}\in\BC$. When setting $\vp_0:=\left\{{2\pi,\atop\vp}\right.{{\rm if\,}\vp=0\atop{\rm else}}$ one finds as an estimate for the error term in Th.\ \ref{AsymptSymmetricSpaces}
\begin{align*}
&\left|\sum_{\a\in\Psi}\bz^R_{\langle\a^\vee,\rho\rangle,\langle\a^\vee,\ell\lambda\rangle,\M}\chi_{\rho+\ell\lambda+k\alpha}(t)\right|
\\&\leq\sum_{\a\in\Psi\atop r,m,\vp}\ell^{-\M}\langle\a^\vee,\lambda\rangle^{-r-\M}m!(r+\M-1)!\left(\frac{2\zeta(2)}{|\vp_0|^{m+r+N+1}}+\frac{\langle\a^\vee,\rho\rangle^{m+r+\M+1}}{(m+r+\M)!}\right)\left|\sum_\psi f_{\a,r,m,\vp,\psi}e^{i\ell\psi}\right|
\\&\leq\frac{(\dim(G/K)^t+\M-1)!}{(2\pi\ell c_2)^\M}\cdot c_1
\end{align*}
with $c_1\in\BR^+$ being independent of $\M$ and $\ell$ and
\[
c_2=\min_{q\in\BZ,w\in W_G,\a\in\Psi}\<\a^\vee,\lambda\>\cdot\left\{
{1\atop|q+w\alpha(X)|}\quad {{\rm if\,}w\a(X)\in\BZ,\atop{\rm else.}}
\right.
\]
\end{lemma}
The last term achieves its minimum at
$\M=\lfloor2\pi\ell c_2\rfloor-\dim(G/K)^t$,
if this value is positive, since the quotient increases when $\M$ increases or decreases from this value. The coefficients $f_{\a,m,r,\vp,\psi}$ are given in terms of $\Sigma_G$ and $W_G$ in \cite[Eq.\ (66)]{K2}. Notice that $m+r\leq\dim (G/T)^t$ holds by same proof as of Prop.\ \ref{AbschDegelldurchDim} when applied to the bundle $L_\lambda^{\otimes\ell}\otimes L_\a^{\otimes k}\to G/T$, as then
\[
\chi_{\rho+\ell\lambda+k\alpha}(t)=\int_{(G/T)^t}\Td_t(T(G/T))\ch_t(L_\lambda^{\otimes\ell})\ch_t(L_\a^{\otimes k}).
\]
\begin{proof}
The first bound follows from Prop.\ \ref{estimate2} when writing $\ell^r$ as $\frac{\langle\a^\vee,\ell\lambda\rangle^r}{\langle\a^\vee,\lambda\rangle^r}$.
Furthermore $r\leq\dim (G/K)^t$ holds by Prop.\ \ref{AbschDegelldurchDim}. Using $|\vp_0|\leq2\pi$ and the Tayler expansion of $\exp$, the bracket in the second line in Lemma \ref{estimate3} can be bounded by
\begin{align*}
\frac{2\zeta(2)}{|\vp_0|^{m+r+N+1}}+\frac{\langle\a^\vee,\rho\rangle^{m+r+\M+1}}{(m+r+\M)!}
&=\frac1{|\vp_0|^{m+r+N+1}}\left(2\zeta(2)+\frac{(|\vp_0|\langle\a^\vee,\rho\rangle)^{m+r+\M+1}}{(m+r+\M)!}\right)
\\&\leq \frac1{|\vp_0|^{m+r+N+1}}\left(2\zeta(2)+2\pi\langle\a^\vee,\rho\rangle e^{2\pi\langle\a^\vee,\rho\rangle}\right).
\end{align*}
Finally by \cite[Eq.\ (66)]{K2} the angles $\vp$ for $\chi_{\rho+\ell\lambda+k\alpha}(t)$ take the form $2\pi w\a(X)$ with $w\in W_G$.
\end{proof}

\section{Comparison with other results}
In this section we give an alternative, elementary proof that shows the compatibility of Th.\ \ref{ErsteTerme} with the results by Bismut--Vasserot and Finski. Suprisingly the proof concerning the summand $\frac7{24}$ in Eq.\ (\ref{BVF2}) turns out to be much more involved than the comparison of the other parts. We need the following three Propositions for this summand.

Let $\<x,y\>_G$ denote the canonical scalar product on $\frak t^*$ associated to $\Sigma$, satisfying $\<x,y\>_G=\sum_{\a\in\Sigma}\<x,\a\>_G\<\a,y\>_G$. This is the scalar product induced by the negative Killing form (\cite[p.\ 214]{BtD}), and for $G$ simple it is proportional to $\<\cdot,\cdot\>$ by Schur's Lemma.
\begin{prop}\label{LinWeyldim}
Let $G$ be a compact Lie group and $V$ an irreducible $G$-representation with character $\chi$ and weights $(\a_j)_{j\in J}$. Then the map
$$
\BZ\to\BQ,\ k\mapsto\sum_{j\in J}\<\a_j,\rho_G\>\prod_{\beta\in\Sigma^+}\frac{\<\beta^\vee,\rho_G+k\a_j\>}{\<\beta^\vee,\rho_G\>}
$$
is linear in $k$.
\end{prop}
This value equals $\sum_{j\in J}\<\a_j,\rho_G\>\dim V_{\rho_G+k\a_j}$ in terms of virtual representations $V_{\rho_G+k\a_j}$.
\begin{proof}
Any $G$-representation can be decomposed into virtual representations where all weights have the same norm, as the subsets of these weights are $W_G$-invariant. Thus without loss of generality we can assume that $\|\a_j\|\equiv{\rm const.}$ and that $G$ is simple.
The virtual representation $\psi^kV$ has weights $(k\a_j)_{j\in J}$ and the character $\oplus_{j\in J}\chi_{\rho_G+k\a_j}$. Then Freudenthal--de Vries' formula \cite[47.10.2]{FrdV} shows that
\begin{eqnarray*}
\sum_{j\in J}\<k\a_j,\rho_G\>_G^2&=&\sum_{j\in J}
\frac{\<k\a_j,2\rho_G+k\a_j\>_G}{24}
\prod_{\beta\in\Sigma^+}\frac{\<\beta,\rho_G+k\a_j\>}{\<\beta,\rho_G\>}
\\&=&\sum_{j\in J}
\frac{k\<\a_j,\rho_G\>_G}{12}
\prod_{\beta\in\Sigma^+}\frac{\<\beta,\rho_G+k\a_j\>}{\<\beta,\rho_G\>}+\frac{k^2}{24}\sum_{j\in J}\|\a_j\|_G^2.
\end{eqnarray*}
In the case of the canonical scalar product the above map is thus equal to
\[
k\cdot\sum_{j\in J}\left(12\<\a_j,\rho_G\>_G^2-\frac{1}2\|\a_j\|_G^2\right).
\qedhere
\]
\end{proof}
\begin{Bem}
By generalising this proof using higher degree Taylor expansions of the Weyl character formula in the approach in \cite[47.10]{FrdV}, one can show that for any $m\in\BN_0$, $\sum_{j\in J}\<\a_j,\rho_G\>^m\prod_{\beta\in\Sigma^+}\frac{\<\beta,\rho_G+k\a_j\>}{\<\beta,\rho_G\>}$ is a polynomial of degree $m$ in $k$ which is alternately odd or even.
\end{Bem}

\begin{prop}\label{VglSymm1}
Consider a compact simple Lie group $G$ and let $G/K$ be an Hermitian symmetric space of complex dimension $n$. Let $\lambda$ be a weight such that $\forall \gamma\in\Sigma_K:\<\gamma,\lambda\>=0$ and $\forall\a\in\Psi:\<\gamma,\lambda\>>0$. Then $\rho-\rho_K=c_\lambda\lambda$, where for any $\a_0\in\Psi$
\begin{equation}
\label{VglSymm11}
c_\lambda=\sqrt{\frac{n}{8}}\frac1{\|\lambda\|_G}=\frac{n}{8\<\lambda,\rho-\rho_K\>_G}=\frac1{4\<\alpha_0,\lambda\>_G}
\end{equation}
and $\|\rho\|_G^2=\|\rho_K\|_G^2+\frac{n}8$.
\end{prop}
\begin{proof}
For any $\gamma\in\Sigma^+_K$, $\sum_{\a\in\Psi}\<\gamma^\vee,\alpha\>=2\<\gamma^\vee,\rho-\rho_K\>=0$, as $s_\gamma\gamma=-\gamma$ and $s_\gamma\in W_K$ permutes $\Psi$. As $K$ equals $S^1$ times a semi-simple Lie group, this implies that $\exists c_\lambda\in\BR:\rho-\rho_K=c_\lambda\lambda$.
The weight $\lambda$ is a multiple of a cominuscule weight $\lambda_0$ in the sense that $\lambda_0^\vee$ is minuscule with respect to $(\Sigma^+)^\vee$ (\cite[Th. 1.25]{Mi}). Thus $\forall\a\in\Psi:\<\lambda,\a\>>0$ implies that $\<\lambda,\a\>$ is independent of $\a\in\Psi$ by \cite[Prop.\ VIII.3.6(iii)]{Bou7-9}. Hence for any $\alpha_0\in\Psi$ we find that $2c_\lambda\|\lambda\|^2=2\<\rho-\rho_K,\lambda\>=\sum_{\alpha\in\Psi}\<\alpha,\lambda\>=n\<\alpha_0,\lambda\>$ and $\frac12\|\lambda\|_G^2=\sum_{\gamma\in\Sigma^+}\<\gamma,\lambda\>_G^2=n\<\alpha_0,\lambda\>_G^2$. This implies Eq.\ (\ref{VglSymm11}).
Furthermore one finds $\|\rho\|_G^2=\|\rho_K\|_G^2+c_\lambda^2\|\lambda\|^2=\|\rho_K\|_G^2+\frac{n}8$.
\end{proof}
In particular $\<\alpha_0,\rho-\rho_K\>_G=c_\lambda\<\alpha_0,\lambda\>_G=\frac14$.
\begin{prop}\label{VglSymm2}
Under the assumptions of Prop.\ \ref{VglSymm1} let $\<x,y\>_K$ denote the canonical scalar product associated to $\Sigma_K$ on a $(\dim\frak t-1)$-dimensional subtorus $\frak t_K^*$. Then 
$$
\frac{\<\cdot,\cdot\>_{G|\frak t_K^*}}{\<\cdot,\cdot\>_K}=\frac{3\dim K-\dim G}{2(\dim K-1)}.
$$
\end{prop}
\begin{proof}
As any $W_G$-invariant scalar product is $W_K$-invariant there exists $c>0$ such that $c{\|\cdot\|_K^2}={\|\cdot\|_G^2}$. Now Freudenthal--de Vries' Strange Formula (\cite[eq. 47.11]{FrdV}) states that $\dim G=24\|\rho_G\|_G^2$ and $\dim K-1=24\|\rho_K\|_K^2$. On the other hand Prop.\ \ref{VglSymm1} shows $\dim K-1=\frac{24}c(\|\rho_G\|_G^2-\frac{n}8)$. Thus $c=\frac{3\dim K-\dim G}{2(\dim K-1)}$.
\end{proof}

\begin{theor}\label{BVFinski}
The leading terms in Theorem \ref{ErsteTerme} equal the terms given by \cite[Th. 8]{BV} and \cite[Th. 1.3]{Fi}. Namely
\begin{eqnarray}\label{BVF1}
T(M,\mtr L^\ell)&=&\int_M\frac{\ell^n}{n!}c_1(L)^n(\frac{n}2\log\frac{\ell}{2\pi}+\frac12\log \mathring \Omega^L)
+o(\ell^{n})
\end{eqnarray}
where $\mathring \Omega^L$ is defined via $\omega^{TM}(\mathring \Omega^L\cdot,\cdot)=\Omega^L$. Furthermore in the case $X_0:=\frac\lambda{2\pi}$ and thus $x_j\equiv 2\pi$, as considered in \cite[Th. 1.3]{Fi},
\begin{eqnarray}\label{BVF2}\nonumber
T(M,\mtr L^\ell)&=&\int_M\Bigg(\frac{\ell^n}{n!}c_1(L)^n\cdot\frac{n}2\log\ell
\\&&\nonumber
+\frac{\ell^{n-1}}{(n-1)!}c_1(L)^{n-1}c_1(TM)\Big(\frac{3n+1}{12}\log\ell
+\frac{24\zeta'(-1)+2\log(2\pi)+7}{24}\Big)\Bigg)
\\&&+o(\ell^{n-1}).
\end{eqnarray}
\end{theor}
\begin{proof}
By the formula for the analytic torsion of products we can assume that $G$ is simple and hence $K$ equals a product of $S^1$ and a semi-simple group (\cite[Prop.\ VIII.6.2]{He}).
We shall use the expansions
\begin{eqnarray*}
\Td(TM)&=&1+\frac{c_1(TM)}2+{\rm higher\ order},
\\ \frac{\Td'(E)}{\Td(E)}-\rk E&=&-\frac{\rk E}2-\frac{c_1(E)}{12}+{\rm higher\ order},
\\ G(TM)&=&\zeta'(0)\rk TM+c_1(TM)\zeta'(-1)+{\rm higher\ order}
\\&=&-\frac{n}2\log(2\pi)+c_1(TM)\zeta'(-1)+{\rm higher\ order},
\\ \ch(L^\ell)&=&\frac{\ell^n}{n!}c_1(L)^n+\frac{\ell^{n-1}}{(n-1)!}c_1(L)^{n-1}+{\rm lower\ order}.
\end{eqnarray*}
Let $T_4$ denote the coefficient of $\ell^{n-1}$ in the summand for $m=n-1$ in Theorem \ref{ErsteTerme}. One finds for the first three summands
\begin{eqnarray}\label{ErsteNichtaeq}
T(G/K,\mtr L^\ell_{\rho_K+\ell\lambda})&=&
\int_M\Bigg(\frac{\ell^n}{n!}c_1(L)^n(\frac{n}2\log\ell-\frac{n}2\log(2\pi)+\sum_{j\in J}\frac{\rk E_j}2\log x_j)
\nonumber\\&&+\frac{\ell^{n-1}}{(n-1)!}c_1(L)^{n-1}\Big(
\frac{c_1(TM)}2(\frac{n}2\log\ell-\frac{n}2\log(2\pi)+\sum_{j\in J}\frac{\rk E_j}2\log x_j)
\nonumber\\&&+(\frac{c_1(TM)}{12}\log\ell+c_1(TM)\zeta'(-1)+\sum_{j\in J}\frac{c_1(E_j)}{12}\log x_j)
\Big)\Bigg)+T_4+o(\ell^{n-1})
\nonumber\\&=&
\int_M\Bigg(\frac{\ell^n}{n!}c_1(L)^n(\frac{n}2\log\frac{\ell}{2\pi}+\sum_{j\in J}\frac{\rk E_j}2\log x_j)
\nonumber\\&&+\frac{\ell^{n-1}}{(n-1)!}c_1(L)^{n-1}\Big(c_1(TM)
(\frac{3n+1}{12}\log\ell-\frac{n}4\log(2\pi)
\nonumber\\&&+\sum_{j\in J}\frac{\rk E_j}4\log x_j
+\zeta'(-1))
+\sum_{j\in J}\frac{c_1(E_j)}{12}\log x_j
\Big)\Bigg)+T_4+o(\ell^{n-1}).
\end{eqnarray}
The curvature term $\mathring \Omega^L$ is a constant times the identity on the irreducible homogeneous space $M$ because of the equivariance. For the choice $X_0=\lambda$ one finds $\frac12\sum_{\a\in\Psi}\log\frac{\langle\a^\vee,\lambda\rangle}{\a^\vee(X_0)}=0$ and $\frac12\log \mathring \Omega^L=0$, as in this case $\omega^{TM}=\Omega^L$. Rescaling $X_0$ on a virtual subrepresentation $E_j$ of $TM$ scales $\mathring \Omega^L$ by the same factor, and thus
$$
\frac12\log \mathring \Omega^L=\frac12\sum_{\a\in\Psi}\log\frac{\langle\a^\vee,\lambda\rangle}{\a^\vee(X_0)}=
\sum_{j\in J}\frac{\rk E_j}2\log x_j.
$$
Thus Eq.\ (\ref{ErsteNichtaeq}) implies Eq.\ (\ref{BVF1}) for symmetric spaces:
$$
\int_Mc_1(L)^n\sum_{j\in J}\frac{\rk E_j}2\log x_j=\int_Mc_1(L)^n\frac12\log \mathring \Omega^L.
$$
Also Eq.\ (\ref{ErsteNichtaeq}) contains the same multiples of $\log\ell$ and $\zeta'(-1)$ as Eq.\ (\ref{BVF2}). Furthermore in the case $x_j\equiv 2\pi$ one finds
\begin{eqnarray*}
&&\int_Mc_1(L)^{n-1}c_1(TM)\frac{2\log(2\pi)}{24}
\\&=&
\int_Mc_1(L)^{n-1}\Big(c_1(TM)
(-\frac{n}4\log(2\pi)
+\sum_{j\in J}\frac{\rk E_j}4\log x_j)
+\sum_{j\in J}\frac{c_1(E_j)}{12}\log x_j
\Big).
\end{eqnarray*}
Finally we show that the term
\begin{eqnarray}\label{mein724}
\nonumber
T_4&=&-\ell^{n-1}\int_{M}\Td(TM)\bz\left(\sum_{\a\in\Psi}\ch(E_{\rho_K+k\alpha})
\frac{c_1(L_{\rho_K+\lambda})^n}{n!}\frac{\langle\a^\vee,\rho\rangle+k}{\langle\a^\vee,\lambda\rangle}
\right)
\nonumber
\\&=&-\frac{\ell^{n-1}}{n!}\bz\left(\sum_{\a\in\Psi}\dim(V^K_{\rho_K+k\alpha})
\frac{\langle\a^\vee,\rho-\rho_K\rangle+\langle\a^\vee,\rho_K\rangle+k}{\langle\a^\vee,\lambda\rangle}\right)\int_{M}c_1(L_{\rho_K+\lambda})^n
\nonumber
\\&=:&T_{41}+T_{42}+T_{43}
\end{eqnarray}
equals $T_F:=\frac{7\ell^{n-1}}{24(n-1)!}\int_Mc_1(L)^{n-1}c_1(TM)$.
First we relate $T_F$ to $\int_Mc_1(L)^n$. The Theorem of Hirzebruch--Riemann--Roch shows that the term of degree $\ell^{n-1}$ in $\sum_{\a\in\Psi}\dim V^G_{\rho+\ell\lambda+k\a}$ is given by
\begin{eqnarray}\nonumber\label{RRln-1}
&&\sum_{\a\in\Psi}(\dim V^G_{\rho+\ell\lambda+k\a})^{[\deg_\ell=n-1]}
=\left(\int_M\Td(TM)\ch(L_{\rho_K+\lambda})^\ell\ch(\psi^kTM)\right)^{[\deg_\ell=n-1]}
\nonumber
\\&=&\int_M(1+\frac12c_1(TM))\frac{\ell^{n-1}c_1(L_{\rho_K+\lambda})^{n-1}}{(n-1)!}(n+kc_1(TM))
\nonumber
\\&=&(n/2+k)\frac{\ell^{n-1}}{(n-1)!}\int_Mc_1(TM)c_1(L_{\rho_K+\lambda})^{n-1}
=\frac{12(n+2k)}7T_F.
\end{eqnarray}
On the other hand the Weyl dimension formula shows
\begin{eqnarray}\label{Vergleich230}
\sum_{\a\in\Psi}(\dim V^G_{\rho+\ell\lambda+k\a})^{[\deg_\ell=n-1]}
&=&\sum_{\a\in\Psi}\left(\prod_{\beta\in\Psi}\frac{\<\beta^\vee,\rho+\ell\lambda+k\a\>}{\<\beta^\vee,\rho\>}\cdot\prod_{\beta\in\Sigma^+_K}\frac{\<\beta^\vee,\rho+k\a\>}{\<\beta^\vee,\rho\>}\right)^{[\deg_\ell=n-1]}
\nonumber
\\&=&\sum_{\a,\gamma\in\Psi}\frac{\<\gamma^\vee,\rho+k\a\>}{\<\gamma^\vee,\ell\lambda\>}
\prod_{\beta\in\Psi}\frac{\<\beta^\vee,\ell\lambda\>}{\<\beta^\vee,\rho\>}\cdot\prod_{\beta\in\Sigma^+_K}\frac{\<\beta^\vee,\rho+k\a\>}{\<\beta^\vee,\rho\>}.
\end{eqnarray}
By Eq.\ (\ref{RRln-1}) this term is affine linear in $k$ and thus it is equal to
\begin{eqnarray}
\sum_{\a\in\Psi}(\dim V^G_{\rho+\ell\lambda+k\a})^{[\deg_\ell=n-1]}
&=&\sum_{\a\in\Psi}\prod_{\beta\in\Psi}\frac{\<\beta^\vee,\ell\lambda\>}{\<\beta^\vee,\rho\>}\cdot\Bigg(\sum_{\gamma\in\Psi}\frac{\<\gamma^\vee,\rho+k\a\>}{\<\gamma^\vee,\ell\lambda\>}
\cdot\prod_{\tilde\beta\in\Sigma^+_K}\frac{\<\tilde\beta^\vee,\rho\>}{\<\tilde\beta^\vee,\rho\>}
\nonumber
\\&&\label{Vergleich23}
+\sum_{\tilde\gamma\in\Psi}\frac{\<\tilde\gamma^\vee,\rho\>}{\<\tilde\gamma^\vee,\ell\lambda\>}
\cdot
\sum_{\gamma\in\Sigma^+_K}\frac{\<\gamma^\vee,k\a\>}{\<\gamma^\vee,\rho\>}\prod_{\beta\in\Sigma^+_K}\frac{\<\beta^\vee,\rho\>}{\<\beta^\vee,\rho\>}\Bigg).
\end{eqnarray}
The relation $\forall\gamma\in\Sigma_K^+:\gamma\perp\rho-\rho_K$ from the proof of Prop.\ \ref{VglSymm1} shows that the second summand in Eq.\ (\ref{Vergleich23}) vanishes and
\begin{eqnarray}\label{Vergleich232}
\sum_{\a\in\Psi}(\dim V^G_{\rho+\ell\lambda+k\a})^{[\deg_\ell=n-1]}
&=&\sum_{\a\in\Psi}\prod_{\beta\in\Psi}\frac{\<\beta^\vee,\ell\lambda\>}{\<\beta^\vee,\rho\>}\cdot\sum_{\gamma\in\Psi}\frac{\<\gamma^\vee,\rho+k\a\>}{\<\gamma^\vee,\ell\lambda\>}
\nonumber
\\&=&\prod_{\beta\in\Psi}\frac{\<\beta^\vee,\ell\lambda\>}{\<\beta^\vee,\rho\>}\cdot\sum_{\gamma\in\Psi}\frac{\<\gamma^\vee,n\rho+2k(\rho-\rho_K)\>}{\<\gamma^\vee,\ell\lambda\>}
.
\end{eqnarray}

Combining Eqs.\ (\ref{RRln-1}) and (\ref{Vergleich232}) implies that
\begin{equation}\label{Vergleich234}
T_F=
\frac{7n\ell^{n-1}c_\lambda}{12}\prod_{\beta\in\Psi}\frac{\<\beta^\vee,\lambda\>}{\<\beta^\vee,\rho\>}.
\end{equation}
By considering the coefficient of $\ell^n$ instead, one obtains $\frac{\ell^n}{n!}\int_Mc_1(L)^n=\prod_{\beta\in\Psi}\frac{\<\beta^\vee,\lambda\>}{\<\beta^\vee,\rho\>}$ the same way as Eq.\ (\ref{Vergleich234}).

The term $T_{41}$:
One finds
\begin{eqnarray*}
T_{41}&=&-\ell^{-1}c_\lambda\bz\left(\sum_{\a\in\Psi}\dim(V^K_{\rho_K+k\alpha})
\right)\prod_{\beta\in\Psi}\frac{\<\beta^\vee,\ell\lambda\>}{\<\beta^\vee,\rho\>}
\\&=&-\ell^{-1}c_\lambda\bz\left(\dim(\psi^k{\rm Ad}^{1,0}_{G/K})
\right)\prod_{\beta\in\Psi}\frac{\<\beta^\vee,\ell\lambda\>}{\<\beta^\vee,\rho\>}
=-\ell^{n-1}c_\lambda\zeta(0)n\prod_{\beta\in\Psi}\frac{\<\beta^\vee,\lambda\>}{\<\beta^\vee,\rho\>}.
\end{eqnarray*}
The term $T_{42}$:
By Proposition \ref{LinWeyldim} the term
$$
\sum_{\a\in\Psi}\dim(V^K_{\rho_K+k\alpha})
\frac{\langle\a^\vee,\rho_K\rangle}{\langle\a^\vee,\lambda\rangle}
=\sum_{\a\in\Psi}\frac{\langle\a^\vee,\rho_K\rangle}{\langle\a^\vee,\lambda\rangle}
\prod_{\beta\in\Sigma^+_K}\frac{\<\beta^\vee,\rho_K+k\alpha\>}{\<\beta^\vee,\rho_K\>}
$$
is linear in $k$ and thus given by
$k\sum_{\a\in\Psi}\frac{\langle\a^\vee,\rho_K\rangle}{\langle\a^\vee,\lambda\rangle}
\sum_{\beta\in\Sigma^+_K}\frac{\<\beta^\vee,\alpha\>}{\<\beta^\vee,\rho_K\>}.$ Therefore
$$
T_{42}=-\ell^{-1}\zeta(-1)\sum_{\a\in\Psi}\frac{\langle\a^\vee,\rho_K\rangle}{\langle\a^\vee,\lambda\rangle}
\sum_{\beta\in\Sigma^+_K}\frac{\<\beta^\vee,\alpha\>}{\<\beta^\vee,\rho_K\>}\cdot\prod_{\beta\in\Psi}\frac{\<\beta^\vee,\ell\lambda\>}{\<\beta^\vee,\rho\>}.
$$

The term $T_{43}$:
The set of weights $\Psi_j:=\{\a\in\Psi\,|\,\<\a,\lambda\>=j\}$ is $W_K$-invariant for any $j\in\BN$. Thus there is a virtual $K$-representation $F_j$ with weights $\Psi_j$, and
\begin{eqnarray*}
\sum_{\a\in\Psi}\frac{\dim V^K_{\rho_K+k\a}}{\<\a^\vee,\lambda\>}=\sum_{j\in\BN}\frac{\dim \psi^kF_j}{j}
\end{eqnarray*}
is independent of $k$. Hence
\begin{eqnarray*}
T_{43}&=&-\ell^{-1}\zeta(-1)\sum_{\a\in\Psi}\frac{1}{\langle\a^\vee,\lambda\rangle}
\prod_{\beta\in\Psi}\frac{\<\beta^\vee,\ell\lambda\>}{\<\beta^\vee,\rho\>}.
\end{eqnarray*}
Dividing $T_F$ and $T_{41}+T_{42}+T_{43}$ by $\frac{\ell^{n-1}c_\lambda}{12}\prod_{\beta\in\Psi}\frac{\<\beta^\vee,\lambda\>}{\<\beta^\vee,\rho\>}$, we see that to prove $T_F=T_4$ we need to show that
\begin{eqnarray*}
7n=
6 n
+\sum_{\a\in\Psi}\frac{\langle\a,\rho_K\rangle}{c_\lambda\<\a,\lambda\>}
\sum_{\beta\in\Sigma^+_K}\frac{\<\beta^\vee,\alpha\>}{\<\beta^\vee,\rho_K\>}
+\sum_{\a\in\Psi}\frac{\|\a\|^2}{2c_\lambda\<\a,\lambda\>}.
\end{eqnarray*}
So far, we have deliberately avoided arguments not applicable to general complex homogeneous spaces. We will now narrow our focus to symmetric spaces. We use $\forall\a\in\Psi:c_\lambda\<\a,\lambda\>_G\equiv\frac14$ as shown by Prop.\ \ref{VglSymm1} and we subtract $6n$ on both sides to simplify the last equation to
\begin{eqnarray}\label{Vergleich235}
n=
4\sum_{\a\in\Psi\atop\beta\in\Sigma^+_K}
\frac{\langle\a,\rho_K\rangle_G\<\beta^\vee,\alpha\>_G}{\<\beta^\vee,\rho_K\>_G}+2\sum_{\a\in\Psi}\|\a\|_G^2.
\end{eqnarray}
Gordon Brown's Formula (\cite{Brown}) shows that $2\sum_{\a\in\Sigma^+}\|\a\|_G^2=\dim\frak t$ and $2\sum_{\a\in\Sigma^+_K}\|\a\|_K^2=\dim\frak t-1$. Furthermore the definition of the canonical scalar product implies that
$$
4\sum_{\a\in\Sigma^+\atop\beta\in\Sigma^+_K}\langle\a,\rho_K\rangle_G
\frac{\<\beta^\vee,\alpha\>_G}{\<\beta^\vee,\rho_K\>_G}=2\sum_{\beta\in\Sigma^+_K}\frac{\<\beta^\vee,\rho_K\>_G}{\<\beta^\vee,\rho_K\>_G}=\#\Sigma_K
=4\sum_{\a\in\Sigma^+_K\atop\beta\in\Sigma^+_K}\langle\a,\rho_K\rangle_K
\frac{\<\beta^\vee,\alpha\>_K}{\<\beta^\vee,\rho_K\>_K}.
$$
Therefore when setting $c:=\frac{\<\cdot,\cdot\>_G}{\<\cdot,\cdot\>_K}$ we get by Prop.\ \ref{VglSymm2}
\begin{eqnarray*}
&&2\sum_{\a\in\Psi}\|\a\|^2_G
+4\sum_{\a\in\Psi\atop\beta\in\Sigma^+_K}\langle\a,\rho_K\rangle_G
\frac{\<\beta^\vee,\alpha\>_G}{\<\beta^\vee,\rho_K\>_G}
=(1-c)\dim\frak t+c+(1-c)\#\Sigma_K
\\&=&\frac{\dim G-\dim K-2}{2(\dim K-1)}(\dim{\frak t}-1+\#\Sigma_K)+1=n.
\end{eqnarray*}
Thus Eq.\ (\ref{Vergleich235}) holds and $T_4=T_F$.
\end{proof}

\section{Application to the Jantzen sum formula}
We give an application of the previous results to lattice representations of Chevalley group schemes by using the relation between analytic torsion and the Jantzen sum formula \cite[p.\ 311]{Jan}. For details on the objects investigated here we have to refer to \cite{Jan} and \cite{KK}.
Consider a semisimple Chevalley group scheme $\cal G$ over ${\rm Spec}\,\BZ$. Fix a maximal split torus ${\cal T}\subseteq \cal G$ with group of characters $X^*({\cal T})$ and an ordering. Let $\cal B$ be an associated Borel subgroup and set $X:={\cal G}/{\cal B}$.
Given two $\BZ$-free $\cal G$-modules $A$, $A'$ such that $A_\BQ$ and $A'_\BQ$ are isomorphic and irreducible, the index $[A,A']\in\BQ^+$ is obtained when
embedding $A'$ into $A\otimes\BQ$ such that the weight spaces of highest weight are identified.

\begin{theor}\label{Jantzen}
Consider a dominant weight $\lambda$ of $\cal G$. Let $w_0$ be the Weyl group element of maximal
length.
Let $A_{\rho+\lambda}:=H^{0}(X,{\cal L}_\lambda)_{\rm free}$ be the $\cal G$-module induced
by a line bundle
${\cal L}_\lambda$. Then for $\ell\to\infty$ one obtains the asymptotic expansion
\begin{align*}
\lefteqn{
\sum_{\mu\in X^*({\cal T})} \mu(t) \log [A_{\rho+\ell\lambda,\mu}:A_{{w_0(\rho+\ell\lambda)},\mu}]
}
\\
=&\sum_{\a\in\Psi}\Bigg(2\widetilde{\chi_{\rho+\ell\lambda-k\alpha}}(t)^*(\langle\alpha^\vee,\rho+\ell\lambda\rangle)+\log\langle\a^\vee,\ell\lambda\rangle\cdot\bz\chi_{\rho+\ell\lambda-k\alpha}(t)
\\
&+\bzs\chi_{\rho+\ell\lambda-k\alpha}(t)
-\bz\left(\chi_{\rho+\ell\lambda+k\alpha}(t)
\cdot\log^\ddagger\Bigg(\frac{\langle\a^\vee,\rho\rangle+k}{\langle\a^\vee,\ell\lambda\rangle}+1\right)
\\&+\chi_{\rho+\ell\lambda}(t)
\cdot\log^\ddagger\left(\frac{\langle\a^\vee,\rho\rangle}{\langle\a^\vee,\ell\lambda\rangle}+1\right)
\Bigg)^{[\deg_\ell>-\M]}
\Bigg)+O(\ell^{-\M}).
\end{align*}
\end{theor}
\begin{proof}
This is obtained by combining the Jantzen sum formula as given in \cite[Cor. 7.3]{KK}
\begin{eqnarray*}
\lefteqn{
\sum_{\mu\in X^*(T)} \mu \log [A_{\rho+\lambda,\mu}:A_{{w_0(\rho+\lambda)},\mu}]
}\\&&
+\sum_{\mu\in X^*(T)} \sum_{q=0}^n \mu\cdot (-1)^q  \Big(\log \# H^q(X,{\cal
L}_\lambda)_{\mu,{\rm tor}}
-\log \#
H^{n-q}(X,{\cal L}_{w_0\ldotp\lambda})_{\mu,{\rm tor}}\Big)
\\&&
-(-1)^{l(w)}\chi_{\rho+\lambda}\sum_{q=0}^n (-1)^q \Big(\log \#
H^q(X,{\cal L}_\lambda)_{\lambda_0,{\rm tor}}-\log
\# H^{n-q}(X,{\cal L}_{w_0\ldotp\lambda})_{{\lambda_0},{\rm tor}}\Big)
\\&
=&-\sum_{\a\in\Psi}\sum_{k=1}^{\langle\a^\vee,\rho+\lambda\rangle-1}
\chi_{\rho+\lambda-k\a}\cdot\log k
\end{eqnarray*}
(slightly simplified for the case $\lambda$ dominant) with Eq.\ (\ref{AsymptSymmetricSpacesProof}). Apply the symmetry (\ref{Lie-symmetry}) as
$$\chi_{\rho+\ell\lambda-\langle\a^\vee,\rho+\ell\lambda\rangle\alpha}(t)\cdot\log \langle\a^\vee,\rho+\ell\lambda\rangle=
-\chi_{\rho+\ell\lambda}(t)\cdot\log \langle\a^\vee,\rho+\ell\lambda\rangle$$
to the term corresponding to $k=\langle\a^\vee,\rho+\ell\lambda\rangle$ in the sum Eq.\ (\ref{AsymptSymmetricSpacesProof}). For $\lambda$ dominant, the Kempf
vanishing theorem
\cite[Prop.\ 4.5]{Jan} shows that the torsion of the cohomology $H^q(X,{\cal L}_\lambda)$ vanishes for all $q$. This implies by \cite[Section 8.8]{Jan} that torsion of the cohomology $H^q(X,{\cal L}_{w_0\lambda})$ vanishes too. This leads to
\begin{align*}
\lefteqn{
\sum_{\mu\in X^*({\cal T})} \mu(t) \log [A_{\rho+\ell\lambda,\mu}:A_{{w_0(\rho+\ell\lambda)},\mu}]
}
\\
=&\sum_{\a\in\Psi}\Bigg(2\widetilde{\chi_{\rho+\ell\lambda-k\alpha}}(t)^*(\langle\alpha^\vee,\rho+\ell\lambda\rangle)-\log\langle\a^\vee,\ell\lambda\rangle\cdot\bz\chi_{\rho+\ell\lambda+k\alpha}(t)
\\
&+\bzs\chi_{\rho+\ell\lambda-k\alpha}(t)
-\chi_{\rho+\ell\lambda}(t)\cdot\log \langle\a^\vee,\rho+\ell\lambda\rangle
\\
&-\bz\left(\chi_{\rho+\ell\lambda+k\alpha}(t)
\cdot\log^\ddagger\left(\frac{\langle\a^\vee,\rho\rangle+k}{\langle\a^\vee,\ell\lambda\rangle}+1\right)\right)^{[\deg_\ell>-\M]}
\\&-\bz^R_{\langle\a^\vee,\rho\rangle,\langle\a^\vee,\ell\lambda\rangle,\M}\chi_{\rho+\ell\lambda+k\alpha}(t)
\Bigg).
\end{align*}
Taylor expanding $\log \langle\a^\vee,\rho+\ell\lambda\rangle\chi_{\rho+\ell\lambda}(t)$ leads to a term $-\log\langle\a^\vee,\ell\lambda\rangle\cdot(\bz\chi_{\rho+\ell\lambda+k\alpha}(t)+\chi_{\rho+\ell\lambda}(t))$, which is simplified by Eq.\ (\ref{zetaSchluckt1}) as
\[
\bz\chi_{\rho+\ell\lambda+k\alpha}(t)+\chi_{\rho+\ell\lambda}(t)=-\bz\chi_{\rho+\ell\lambda+k\alpha}(t).
\qedhere
\]
\end{proof}
This result can be rewritten using $\Phi(e^{i\vp},-m,0)=e^{i\vp}\Phi(e^{i\vp},-m,1)+\left\{
{1\atop0}\mbox{ if }{m=0,\atop m>0}
\right.$ by Eq.\ (\ref{Apostol0}) and by further simplifying the error term.

\section{The case of isolated fixed points}\label{SectionHomom}
In this chapter $G/K$ can be any compact complex homogeneous manifold or, in other words, a (generalised) flag manifold. Set ${\frak t}_{\rm reg}:=\{X\in\frak t|\a(X)\notin \BZ \
\forall\a\in\Sigma_G\}$. We shall use the operators $\bz_{m,a}$ introduced in Def. \ref{zetama}.
Notice that $\bz_{0,a}e^{ik\vp}\equiv\bz e^{ik\vp}=\frac1{e^{-i\vp}-1}$ by Eq.\ (\ref{IterationPhi}).

\begin{theor}\label{isolFixpkt}
Let $M=G/K$ be a compact complex homogeneous manifold and let $\mtr L:=\mtr L_{\rho_K+\lambda}$ be a $G$-invariant holomorphic Hermitian line bundle on $M$. Assume that $\mtr L$ is positive in the sense that $\forall\alpha\in\Psi:\<\alpha^\vee,\lambda\>>0$. Fix a K\"ahler metric $g_{X_0}$ on $M$. Choose $X\in{\frak t}_{\rm reg}$ and set $t:=e^X$. Then the asymptotic behaviour of the equivariant analytic torsion for high powers of $\mtr L$ is given by
\begin{eqnarray*}
\lefteqn{
T_t((M,g_{X_0}),\mtr L_{\rho_K+\lambda}^\ell)
=
-\log\ell\cdot\bz\sum_{\alpha\in\Psi}\chi_{\rho+\ell\lambda+k\alpha}(t)
+\bzs\sum_{\alpha\in\Psi}\chi_{\rho+\ell\lambda+k\alpha}(t)
}\\&&-\sum_{\alpha\in\Psi} \log\frac{\langle\a^\vee,\lambda\rangle}{\alpha^\vee(X_0)}\cdot\bz
\chi_{\rho+\ell\lambda+k\alpha}(t)
+\tilde C\cdot
\chi_{\rho+\ell\lambda}(t)
\\&&+\sum_{\a\in\Psi}\sum_{m=1}^{\M-1}\frac{\bz_{m,\langle\a^\vee,\rho\rangle}\chi_{\rho+\ell\lambda+k\alpha}(t)}{(-\langle\a^\vee,\ell\lambda\rangle)^{m}m}
-\sum_{\a\in\Psi}\bz^R_{\langle\a^\vee,\rho\rangle,\langle\a^\vee,\lambda\rangle\ell,\M}\chi_{\rho+\ell\lambda+k\alpha}(t)
\\&=&\log\ell\cdot \int_{M^t}
\Td^*_t(TM)\ch_t(L_{\rho_K+\lambda}^\ell)+O(1).
\end{eqnarray*}
The constant $\tilde C$ vanishes if $G$ is not of type $G_2$, $F_4$ or $E_8$.
\end{theor}
Here $\chi_{\rho+\ell\lambda+k\alpha}(t)$ is a linear combination of terms of the form $e^{i(\vp'\ell+\vp k)}$ with $\vp,\vp'\in\BR\setminus2\pi\BZ$. 
\begin{proof}
By \cite[Th. 5.4]{KK}, by the formula for the constant $C$ in \cite[Th. 7.1]{KK} for positive $\mtr L$ and by the remark at the end of \cite[p.\ 660]{KK}, the torsion is given by
\begin{eqnarray*}
\lefteqn{T_t((M,g_{X_0}),\mtr L_{\rho_K+\lambda}^\ell)=
\int_{M^t} \Td_t(TM)R_t(TM)\ch_t(L_{\rho_K+\lambda}^\ell)
}\\
&&-\sum_{\alpha\in\Psi} \bz
\chi_{\rho+\ell\lambda-k\alpha}(t)\cdot\log\alpha^\vee(X_0)
-
\chi_{\rho+\ell\lambda}(t)\sum_{\a\in\Psi^+}\log\a^\vee(X_0)
-\sum_{\a\in\Psi}\sum_{k=1}^{\langle\a^\vee,\rho+\ell\lambda\rangle}
\chi_{\rho+\ell\lambda-k\alpha}(t)\cdot\log k.
\end{eqnarray*}
The result thus follows the same way as in the proof of Theorem \ref{AsymptSymmetricSpaces} and Theorem \ref{ErsteTerme} in the case of Hermitian symmetric spaces, as the formula for the analytic torsion has the same form. In the case $X\in{\frak t}_{\rm reg}$ the polynomial degree of $\chi_{\rho+\ell\lambda+k\alpha}(t)$ in $\ell$ and $k$ equals 0, and thus the $\log^\ddagger$-summand in Theorem \ref{AsymptSymmetricSpaces} is given by the result above.
\end{proof}
\begin{lemma}\label{errortermIsol}
The error term in Theorem \ref{isolFixpkt} can be estimated as
\begin{align*}
\lefteqn{
\left|
\sum_{\a\in\Psi}\bz^R_{\langle\a^\vee,\rho\rangle,\langle\a^\vee,\lambda\rangle\ell,\M}\chi_{\rho+\ell\lambda+k\alpha}(e^X)
\right|<
\frac{\ell^{-N}}{\left|\prod_{\a\in \Sigma^+}2\sin(\pi iw\a(X))\right|}
}
\\&\cdot\sum_{\a\in\Psi}\frac{1}{\langle\a^\vee,\lambda\rangle^{\M}}\Bigg(
\frac{(\M-1)!}{(2\pi)^{\M+1}}\sum_{w\in W_G}(\zeta(\M+1,w\alpha(X))+\zeta(\M+1,1-w\alpha(X)))
\\&+\frac{\#W_G}{\M}\cdot (\zeta(-\M)-\zeta(-\M,\langle\a^\vee,\rho\rangle+1))
\Bigg)
\\<&
\frac{\ell^{-N}}{\left|\prod_{\a\in \Sigma^+}2\sin(\pi iw\a(X))\right|}
\\&\cdot\sum_{\a\in\Psi}\frac{1}{\langle\a^\vee,\lambda\rangle^{\M}}\Bigg(
(\M-1)!\sum_{w\in W_G}\frac{2\zeta(2)}{(2\pi \min_{q\in\BZ}|q+w\alpha(X)|)^{N+1}}
\\&+\frac{\#W_G}{\M}\cdot \langle\a^\vee,\rho\rangle^{\M+1}
\Bigg)
\\\leq&\frac{(\M-1)!}{(2\pi\ell\min_{q\in\BZ,w\in W_G,\a\in\Psi}|q+w\alpha(X)|\<\a^\vee,\lambda\>)^\M}\cdot c_1
\end{align*}
with $c_1\in\BR^+$ being independent of $\M$ and $\ell$.
\end{lemma}
\begin{proof}
Since $X\in{\frak t}_{\rm reg}$, the characters are given by
$$
\chi_{\rho+\ell\lambda+k\alpha}(e^X)=\frac{\sum_{w\in W_G}(-1)^w\exp(2\pi iw(\rho+\ell\lambda+k\alpha)(X))}{\prod_{\a\in \Sigma^+}2i\sin(\pi iw\a(X))}.
$$
Theorem \ref{AsymptLerch} shows that
\begin{eqnarray*}
\lefteqn{
\left|\bz^R_{\langle\a^\vee,\rho\rangle,\langle\a^\vee,\lambda\rangle\ell,\M}e^{2\pi ikw\alpha(X)}
\right|
<\frac{\ell^{-\M}}{\M}C(e^{2\pi ikw\alpha(X)},-\M,\langle\a^\vee,\rho\rangle+1)
}
\\&=&\ell^{-\M}\bigg(\frac{(\M-1)!}{(2\pi)^{\M+1}}(\zeta(\M+1,w\alpha(X))+\zeta(\M+1,1-w\alpha(X))
\\&&+\frac{\zeta(-\M)-\zeta(-\M,\langle\a^\vee,\rho\rangle+1}{\M}\bigg).
\end{eqnarray*}
Combining these we obtain the first estimate. The other estimates follow as in Lemma \ref{estimate3}.
\end{proof}

\begin{ex}\label{P1Beisp} On $\BP^1\BC=\SU(2)/{\bf S}({\bf U}(1)\x {\bf U}(1))$, let $\lambda$ be the weight given by $\left({i\vp\atop0}{0\atop -i\vp}\right)\mapsto \frac\vp{2\pi}$. According to \cite[Th. 2]{K1}, the torsion is for any $\ell\in\BZ$ given by the real number
\begin{eqnarray*}
T_t((\BP^1\BC,g_{X_0}),\mtr L_{\rho_K+\lambda}^\ell)
&=&
\frac{\cos|\ell+1|\frac\vp2}{i\sin\frac\vp2}\left(e^{i\vp}\frac{\partial\Phi}{\partial s}(e^{i\vp},0,1)-e^{-i\vp}\frac{\partial\Phi}{\partial s}(e^{-i\vp},0,1)\right)
\\&&-\sum_{k=1}^{|\ell+1|}\frac{\sin(|\ell+1|-2k)\frac\vp2}{\sin\frac\vp2}\log k.
\end{eqnarray*}
for the Fubini-Study metric with $\|\alpha\|^2=2$. For general metrics there is an additional term which can be written as $\frac{\cos(\ell+2)\frac\phi2}{2\sin^2\frac\phi2}\log\alpha^\vee(X_0)$ by \cite[Th. 5.2]{KK}.
\end{ex}
On the other hand, Th.\ \ref{isolFixpkt} shows in this case:
\begin{cor}\label{asymptTP1}
The equivariant torsion $T_t(\BP^1\BC,{\cal O}(\ell))$ for the rotation $t$ of the sphere by an angle $\vp$ has for $\ell\to+\infty$ the asymptotic expansion
\begin{eqnarray*}
T_t(\BP^1\BC,{\cal O}(\ell))&=&
\frac{\cos(\ell+2)\frac{\vp}2}{-2 \sin^2\frac{\vp}2}\log\ell
\\&&+\frac{\cos(\ell+2)\frac{\vp}2}{2 \sin^2\frac{\vp}2}\log\alpha^\vee(X_0)
+\frac{\cos |\ell+1|\frac{\vp}2}{\sin\frac{\vp}2}R^{\rm rot}(\vp)
\\&&+\frac{\sin|\ell+1|\frac{\vp}2}{2\sin\frac{\vp}2}\Bigg(
-\log2\pi+\Gamma'(1)-\frac{\psi(\frac{\vp}{2\pi})+\psi(1-\frac{\vp}{2\pi})}2
\Bigg)
\\&&+\frac1{2i\sin\frac\vp2}\Bigg(
\sum_{m=1}^{\M-1}\frac{e^{i(\ell+3)\frac\vp2}\Phi(e^{i\vp},-m,2)-e^{-i(\ell+3)\frac\vp2}\Phi(e^{-i\vp},-m,2)}{(-\ell)^{m}m}
\\&&+\frac{e^{i(\ell+3)\frac\vp2}R(e^{i\vp},-\M,2,0,\ell)
-e^{-i(\ell+3)\frac\vp2}R(e^{-i\vp},-\M,2,0,\ell)}{(-\ell)^\M\M}
\Bigg).
\end{eqnarray*}
By applying $T_t(\BP^1\BC,{\cal O}(-\ell))=T_t(\BP^1\BC,{\cal O}(\ell-2))$ this formula also yields a slightly different asymptotic expansion for $\ell\to-\infty$.
\end{cor}
By Lemma \ref{errortermIsol} the last summand is bounded by
$$
\frac{\frac{\M!}{(2\pi)^{\M+1}}(\zeta(\M+1,\frac\vp{2\pi})+\zeta(\M+1,1-\frac\vp{2\pi}))+1}{\ell^\M\M\cdot|\sin\frac\vp2|}.
$$
\begin{proof}
In this case $\Psi=\{2\lambda\}$ and $\chi_{\rho+\ell\lambda+k\alpha}(\left({e^{i\vp/2}\atop0}{0\atop e^{-i\vp/2}}\right))=\frac{\sin(\ell+2k+1)\frac\vp2}{\sin\frac\vp2}$ holds for the rotation of the sphere by an angle $\vp$. We assume that $\ell>0$. The coefficient of $\log\ell$ is thus given by
\begin{eqnarray*}
-\bz\frac{\sin(\ell+2k+1)\frac\vp2}{\sin\frac\vp2}
=\frac{-1}{2i\sin\frac\vp2}\left(\frac{e^{i(\ell+1)\vp/2}}{e^{-i\vp}-1}-\frac{e^{-i(\ell+1)\vp/2}}{e^{i\vp}-1}\right)
=\frac{\cos((\ell+2)\frac{\vp}2)}{-2 \sin^2\frac\vp2}.
\end{eqnarray*}
Equivalently,
\begin{eqnarray*}
\int_{M^t}
\Td^*_t(TM)\ch_t(L_{\rho_K+\lambda}^\ell)
=\frac{e^{i\ell\vp/2}}{(1-e^{-i\vp})^2}+\frac{e^{-i\ell\vp/2}}{(1-e^{i\vp})^2}
=\frac{\cos((\ell+2)\frac{\vp}2)}{-2 \sin^2\frac\vp2}.
\end{eqnarray*}
Therefore this term equals also the coefficient of $-\log\alpha^\vee(X_0)$.
Hence for $\ell\to+\infty$
\begin{eqnarray*}
\lefteqn{
T_t((M,g_{X_0}),\mtr L_{\rho_K+\lambda}^\ell)
=
-\frac{\cos((\ell+2)\frac{\vp}2)}{2 \sin^2\frac\vp2}\cdot\log\frac{\ell}{\alpha^\vee(X_0)}
}\\&&
+\frac1{2i\sin\frac\vp2}\Bigg(e^{i(\ell+3)\frac\vp2}\frac{\partial\Phi}{\partial s}(e^{i\vp},0,1)
-e^{-i(\ell+3)\frac\vp2}\frac{\partial\Phi}{\partial s}(e^{-i\vp},0,1)
\\&&+\sum_{m=1}^{\M-1}\frac{e^{i(\ell+3)\frac\vp2}\Phi(e^{i\vp},-m,2)-e^{-i(\ell+3)\frac\vp2}\Phi(e^{-i\vp},-m,2)}{(-\ell)^{m}m}
\\&&+\left(\bz^R_{1,\ell,\M}e^{-i(\ell+1+2k)\frac\vp2}-\bz^R_{1,\ell,\M}e^{i(\ell+1+2k)\frac\vp2\vp}\right)
\Bigg).
\end{eqnarray*}
Eq.\ (\ref{Rrot-Formel}) provides the value of the $\bzs$-term.
\end{proof}

\section{Lie-algebra-equivariant torsion on the projective plane}
An analytic torsion which is equivariant with respect to the action of a Lie algebra has been defined in \cite[Section 2.6]{BG}.
Let $X$ be a generator of an isometric circle action on $\BP^1\BC$ such that the its flow has period $2\pi$.
By \cite[Th. 9.4]{KSgenus} Bismut--Goette's $X$-equivariant torsion is given by
\begin{eqnarray}\label{gEquivTorsion}\nonumber
T_{tX}(\BP^1\BC,{\cal O}(\ell))&=&
-\frac{\cos\frac{(\ell+1)t}2}{\sin\frac{t}2}\sum_{k\geq1\atop k{\rm\,odd}}\left(2\zeta'(-k)
+{\cal H}_k\zeta(-k)\right)\frac{(-1)^{\frac{k+1}2}t^k}{k!}
\\&&+\sum_{k=1}^{|\ell+1|}\frac{\sin(2k-|\ell+1|)\frac{t}2}{\sin\frac{t}2}\log k
+\left(
\frac{\cos\frac{(\ell+1)t}2)}{t\sin\frac{t}2}
\right)^\#
\end{eqnarray}
where $(t^{2k})^\#:=t^{2k}\cdot\left\{
{2{\cal H}_{2k+1}-{\cal H}_{k}\atop0}
\mbox{ if }{k\geq0,\atop k=-1}
\right.$.
This expression can be interpreted in two ways, which give different asymptotic expansions in $\ell$: As a convergent power series in $t$ for $|t|<2\pi$, and as formal power series in $t$, where one gets one asymptotic expansion for each power $t^m$. In this section we shall describe the asymptotic expansion with respect to the first interpretation. Section \ref{TorsionForm} shall use the other interpretation.

We give the asymptotic expansion in terms of $\tilde\ell:=|\ell+1|$. Thus the coefficients of the two leading terms are the same as for an asymptotic expansion in terms of $\ell$. Using $(\ell+1)^{-m}\sim\sum_{k=0}^\infty\left({m+k-1\atop k}\right)(-1)^k\ell^{-m-k}$ and $\log(\ell+1)\sim\log\ell-\sum_{k=1}^\infty\frac{(-1)^k}{k\ell^k}$ one can obtain the asymptotic expansions in $\ell$.
\begin{theor}\label{AsymptP1Lie}
Define $f(r):=\left(
\frac{r}{\sin\frac{tr}2}-\frac{1}{\sin\frac{t}2}
\right)\cdot\frac{1}{t(1-r^2)}$ as the continuous extension to $r\in]-\frac{2\pi}{|t|},\frac{2\pi}{|t|}[$.
The summands of $T_{tX}(\BP^1\BC,{\cal O}(\ell))$ have the following asymptotic expansion for $\tilde\ell:=|\ell+1|\to\infty$:
\begin{eqnarray*}
\lefteqn{
\sum_{k=1}^{\tilde\ell}\frac{\sin(2k-\tilde\ell)\frac{t}2}{\sin\frac{t}2}\log k=
-\frac{\cos\frac{(\tilde\ell+1)t}{2}}{2\sin^2\frac{t}2}\log\tilde\ell
}\\&&-\frac1{2i\sin\frac{t}2}\left(e^{\frac{-i(\tilde\ell-2)t}2}\frac\partial{\partial s}\Phi(e^{it},0,1)
-e^{\frac{i(\tilde\ell-2)t}2}\frac\partial{\partial s}\Phi(e^{-it},0,1)\right)
\\&&+\frac1{2i\sin\frac{t}2}\sum_{m=1}^{\M-1}\frac{(-\tilde\ell)^{-m}}m 
\left(e^{\frac{i(\tilde\ell+2)t}2}\Phi(e^{it},-m,1)
-e^{-\frac{i(\tilde\ell+2)t}2}\Phi(e^{-it},-m,1)\right)
\\&&+\frac1{2i\sin\frac{t}2}\frac{(-\tilde\ell)^{-\M}}\M 
\left(e^{\frac{i(\tilde\ell+2)t}2}R(e^{it},-\M,1,0,\tilde\ell)
-e^{-\frac{i(\tilde\ell+2)t}2}R(e^{-it},-\M,1,0,\tilde\ell))\right)
\end{eqnarray*}
and
\begin{eqnarray*}
\left(
\frac{\cos\frac{\tilde\ell t}2}{t\sin\frac{t}2}
\right)^\#&\sim&
-\frac1{t\sin\frac{t}2}\Bigg(\sin\frac{\tilde\ell t}2\cdot
\left(\frac\pi2+\sum_{m=0}^\infty\frac{(-1)^m(2m)!}{(\tilde\ell t)^{2m+1}}\right)
\\&&+\cos\frac{\tilde\ell t}2\cdot\left(-\log(\tilde\ell t)+\Gamma'(1)+\sum_{m=1}^\infty\frac{(-1)^m(2m-1)!}{(\tilde\ell t)^{2m}}
\right)\Bigg)
\\&&+\sum_{m=1}^\infty2\left(\frac2{\tilde\ell t}\right)^m\cos\frac{\tilde\ell t+\pi m}2\cdot f^{(m-1)}(1).
\end{eqnarray*}
\end{theor}
\begin{Bem}$t\sin(\frac{t}2)^{m+1}f^{(m-1)}(1)$ is a polynomial in $t$, $\sin\frac{t}2$ and $\cos\frac{t}2$. Using the series $\frac1{\sin\frac{t}2}=\frac2t+\frac{t}{\pi^2}\sum_{m=1}^\infty\frac{(-1)^m}{(\frac{t}{2\pi})^2-m^2}$ one finds
\[
f^{(m-1)}(1)=(m-1)!t^{m-1}\sum_{p=1}^\infty\frac{(-1)^p4p\pi}{4p^2\pi^2-t^2}\left(
\frac1{(2p\pi-t)^{m}}-\frac1{(-2p\pi-t)^{m}}
\right).
\]
\end{Bem}
\begin{proof}
The first expansion follows by Proposition \ref{AsympMitZeta} as in Corollary \ref{asymptTP1}.
By \cite[Prop.\ 9.2]{KSgenus} we know that
$$
\left(
\frac{\cos\frac{\tilde\ell t}2}{t\sin\frac{t}2}
\right)^\#
=
-\int_{-1}^{1}
\left(
\frac{ r \cos\frac{\tilde\ell tr}2}{\sin\frac{tr}2}
-\frac{\cos\frac{\tilde\ell t}2}{\sin\frac{t}2}
\right)\cdot\frac{dr}{t(1-r^2)}.
$$
Furthermore \cite[Th. 9.3]{KSgenus} explains how to find the asymptotic expansion of this integral: We shall need the sine and cosine integral functions defined by ${\rm Si}(x)=\int_0^x\frac{\sin r}{r}\,dr$ and ${\rm Ci}(x)=-\int_x^{+\infty}\frac{\cos r}{r}\,dr$ for $x\in\BR^+$. In the proof of \cite[Th. 9.3]{KSgenus} the above integral is decomposed as
\begin{equation}
\label{SKlasseAsympt}
\left(
\frac{\cos\frac{\tilde\ell t}2}{t\sin\frac{t}2}
\right)^\#
=-\int_{-1}^1\cos\frac{\tilde\ell tr}2\cdot f(r)\,dr
-\frac{\sin\frac{\tilde\ell t}2\cdot {\rm Si}(\tilde\ell t)
-\cos\frac{\tilde\ell t}2\cdot\left(-\Gamma'(1)-{\rm Ci}(\tilde\ell t)+\log(\tilde\ell t)
\right)}
{t\sin\frac{t}2}.
\end{equation}
Partial integration shows for the $2m$-th derivative of $f$ that
\begin{eqnarray*}
\int_{-1}^1\cos\frac{\tilde\ell tr}2\cdot f^{(2m)}(r)\,dr&=&\frac2{\tilde\ell t}\sin\frac{\tilde\ell tr}2\cdot f^{(2m)}(r)\Big|_{-1}^1
-\frac2{\tilde\ell t}\int_{-1}^1\sin\frac{\tilde\ell tr}2\cdot f^{(2m+1)}(r)\,dr
\\&\stackrel{f\,\rm even}=&2\frac2{\tilde\ell t}\sin\frac{\tilde\ell t}2\cdot f^{(2m)}(1)
-\frac2{\tilde\ell t}\int_{-1}^1\sin\frac{\tilde\ell tr}2\cdot f^{(2m+1)}(r)\,dr
\end{eqnarray*}
and analogously
\begin{eqnarray*}
\int_{-1}^1\sin\frac{\tilde\ell tr}2\cdot f^{(2m+1)}(r)\,dr=
-2\frac2{\tilde\ell t}\cos\frac{\tilde\ell t}2\cdot f^{(2m+1)}(1)
+\frac2{\tilde\ell t}\int_{-1}^1\cos\frac{\tilde\ell tr}2\cdot f^{(2m+2)}(r)\,dr.
\end{eqnarray*}
Combining these one obtains for any $N\in\BN_0$ by induction
\begin{eqnarray*}
-\int_{-1}^1\cos\frac{\tilde\ell tr}2\cdot f(r)\,dr
&=&\sum_{m=1}^{N}2\left(\frac2{\tilde\ell t}\right)^m(-1)^{m(m+1)/2}f^{(m-1)}(1)\cdot
\left\{
{\sin\frac{\tilde\ell t}2\atop\cos\frac{\tilde\ell t}2}
\mbox{ for $m$ }{\mbox{odd,}\atop\mbox{even}}
\right.
\\&&-\left(\frac2{\tilde\ell t}\right)^{N}(-1)^{N(N+1)/2}\cdot\int_{-1}^1dr\ f^{(N)}(r)
\left\{
{\sin\frac{\tilde\ell tr}2\atop\cos\frac{\tilde\ell tr}2}
\mbox{ for $N$ }{\mbox{odd,}\atop\mbox{even.}}
\right.
\end{eqnarray*}
Using
\[
\cos\frac{\tilde\ell t+\pi m}2=
(-1)^{m(m+1)/2}f^{(m-1)}(1)\cdot
\left\{
{\sin\frac{\tilde\ell t}2\atop\cos\frac{\tilde\ell t}2}
\mbox{ for $m$ }{\mbox{odd,}\atop\mbox{even}}
\right.
\]
the last term is further simplified.
The trigonometric integrals have the asymptotic expansions (\cite[9.8(10)-(11)]{Erd2})
\begin{eqnarray*}
{\rm Si}(\tilde\ell t)&\sim&\frac\pi2+\sin \tilde\ell t\cdot\sum_{m=1}^\infty\frac{(-1)^m(2m-1)!}{(\tilde\ell t)^{2m}}-\cos \tilde\ell t\cdot\sum_{m=0}^\infty\frac{(-1)^m(2m)!}{(\tilde\ell t)^{2m+1}},
\\{\rm Ci}(\tilde\ell t)&\sim&\cos \tilde\ell t\cdot\sum_{m=1}^\infty\frac{(-1)^m(2m-1)!}{(\tilde\ell t)^{2m}}+\sin \tilde\ell t\cdot\sum_{m=0}^\infty\frac{(-1)^m(2m)!}{(\tilde\ell t)^{2m+1}}
\end{eqnarray*}
and thus one obtains the expansion of the second summand in Eq.\ (\ref{SKlasseAsympt}) by trigonometric addition formulae.
\end{proof}

\section{Asymptotics of torsion forms}\label{TorsionForm}
In this section the asymptotic expansion of torsion forms, as defined in \cite[Def. 3.8]{BK}, for fibrations by projective lines is computed. This is achieved using the formula for this specific torsion form $T_\pi(\mtr{{\cal O}(\ell)})$ given by the author as \cite[Th. 1.1]{KSgenus}. We refer to that article for details on the definition of this object. In the computation we will need the following result that apparently has not been published before.

\begin{prop}\label{nichtHa}
For any $m\in\BZ^+$, $j\in\BZ^+_0$, $a\in\BC$ the following identity holds:
\begin{eqnarray*}
\sum_{r=0}^j\left({j\atop r}\right)\frac{a^r}{j+m-r}=\frac1{m\left({j+m\atop m}\right)}\left(
(-1)^ma^{m+j}+(a+1)^{j+1}\sum_{r=0}^{m-1}\left({j+r\atop j}\right)(-a)^{m-1-r}
\right).
\end{eqnarray*}
\end{prop}
Hansen gives a recursion relation for the left hand side in \cite[(6.11.19)]{Hansen} and provides values for a few special values of $m$. The term $\sum_{r=0}^{m-1}(-1)^{r+m+1}\left({j+r\atop j}\right)a^{m-1-r}$ is the regular part of the Laurent series about $a=0$  of $-\frac{(-1)^ma^{m+j}}{(a+1)^{j+1}}$ outside the unit disc. For the case $m+j<0$ see Eq.\ (\ref{nichtHa2}).
\begin{proof}
In the case $a=0$, both side equal $\frac1{j+m-r}$. Now 
replace $a$ by $-\frac1a$ and $r$ on the left side by $j-r$ to obtain
\begin{eqnarray*}
(-a)^{-m-j}\sum_{r=0}^j\left({j\atop r}\right)\frac{(-a)^{m+r}}{m+r}=\frac1{m\left({j+m\atop m}\right)}\left(
(-1)^m(-a)^{-m-j}+(1-\frac1a)^{j+1}\sum_{r=0}^{m-1}\left({j+r\atop j}\right)a^{r+1-m}
\right).
\end{eqnarray*}
Multiplying by $(-a)^{m+j}$ transforms this equation to
\begin{eqnarray}\label{nichtHa2}
\sum_{r=0}^j\left({j\atop r}\right)\frac{(-a)^{m+r}}{m+r}=\frac{(-1)^m}{m\left({j+m\atop m}\right)}\left(
1-(1-a)^{j+1}\sum_{r=0}^{m-1}\left({j+r\atop j}\right)a^r
\right).
\end{eqnarray}
Differentiating the right hand side of Eq.\ (\ref{nichtHa2}) by $a$ yields
\begin{eqnarray*}
&&\frac{(-1)^m}{m\left({j+m\atop m}\right)}\Bigg(
(1-a)^j\sum_{r=0}^{m-1}\left({j+r\atop j}\right)(j+1)a^r
-(1-a)^j(1-a)\sum_{r=0}^{m-1}\left({j+r\atop j}\right)ra^{r-1}
\Bigg)
\\&=&\frac{(-1)^m(1-a)^j}{m\left({j+m\atop m}\right)}\Bigg(\sum_{r=0}^{m-1}\Bigg(\left({j+r\atop j}\right)(j+1)
-\left({j+r+1\atop j}\right)(r+1)
\\&&+\left({j+r\atop j}\right)r\Bigg)a^r
+m\left({j+m\atop m}\right)a^{m-1}\Bigg)
\\&=&-(1-a)^j(-a)^{m-1}
\end{eqnarray*}
as every summand in the sum over $r$ vanishes. The result equals the derivative of the left hand side of Eq.\ (\ref{nichtHa2}). As both sides of Eq.\ (\ref{nichtHa2}) equal 0 at $a=0$, the formula holds for any $a$.
\end{proof}
\begin{Bem}
Multiplying both sides by $m\left({j+m\atop m}\right)$ leads to $\sum_{r=0}^j\left({j+m\atop m}\right)\left({j\atop r}\right)\frac{ma^r}{j+m-r}$ on the left hand side. Expanding the term $(a+1)^{j+1}$ on the right hand side gives
\begin{eqnarray*}
&&
(-1)^ma^{m+j}+\sum_{q=0}^{j+1}\sum_{r=0}^{m-1}(-1)^{r+m+1}\left({j+1\atop r}\right)\left({j+r\atop j}\right)a^{m-1-r+q}
\\&\stackrel{r\mapsto \atop m-1-r+q}=&
(-1)^ma^{m+j}+\sum_{q=0}^{j+1}\sum_{r=q}^{q+m-1}(-1)^{r+q}\left({j+1\atop r}\right)\left({j+m-1-r+q\atop j}\right)a^r
\\&=&
(-1)^ma^{m+j}+\sum_{r=0}^{j+m}\sum_{q=r}^{r+m-1}(-1)^{r+q}\left({j+1\atop r}\right)\left({j+m-1-r+q\atop j}\right)a^r
\\&\stackrel{q\mapsto r-q}=&
(-1)^ma^{m+j}+\sum_{r=0}^{j+m}\sum_{q=0}^{m-1}(-1)^{q}\left({j+1\atop r-q}\right)\left({j+m-1-q\atop j}\right)a^r
\end{eqnarray*}
and thus comparing the coefficients of powers of $a$ shows for any $m\in\BZ^+$, $j\in\BZ^+_0$, $0\leq r< m+j$ that
\begin{equation}\label{nichtHaBinomial}
m\left({j+m\atop m}\right)\left({j\atop r}\right)=(j+m-r)\sum_{q=0}^{m-1}(-1)^{q}\left({j+1\atop r-q}\right)\left({j+m-1-q\atop j}\right).
\end{equation}
Clearly this formula holds for $m=0$ and any $r\in\BZ$ as well. For fixed $r\in\BZ^+_0$ both sides are polynomials in $j$ and thus Eq.\ (\ref{nichtHaBinomial}) holds for any $m\in\BZ^+_0,r\in\BZ^+_0,j\in\BC$, which in turn can be used to generalize Prop.\ \ref{nichtHa}.
\end{Bem}
\begin{theor}\label{asympTpi}
Let $P\to B$ be an ${\bf U}(2)$ principal bundle, $\pi:\BP\to B$ the induced $\BP^1\BC$-bundle and $E:=P\x_{{\bf U}(2)}\BC^2$, thus $\BP=\BP E$. 
The asymptotic expansion of the torsion form $T_\pi(\mtr{{\cal O}(\ell)})$ for $\mtr {\cal O}(\ell)$ on the $\BP^1\BC$-bundle $\pi:\BP E\to B$ for $\ell\to\infty$ in any differential form degree $n\in\BN_0$ is given by
\begin{eqnarray*}
T_\pi(\mtr{{\cal O}(\ell)})^{[2n]}
&=&\sum_{m=0}^{\lfloor n/2\rfloor}
\frac{2c_1(E)^{n-2m}(c_1(E)^2-4c_2(E))^m}{(n-2m)!(2m+1)!}(-\frac\ell2)^{n-2m}
\\&&\cdot\Bigg(\sum_{k=0}^m
\left(\zeta'(-2k-1)
-{\cal H}_{2k+1}\frac{B_{2k+2}}{2k+2}\right)
\left({2m+1\atop2k+1}\right)B_{2m-2k}\left(-\frac\ell2\right)
\\&&-B_{2m+2}\left(-\frac\ell2\right)\frac{2{\cal H}_{2m+1}-{\cal H}_{m}}{2m+2}
\\&&-\frac1{2m+2}\sum_{j=1}^{2m+2}\left({2m+2\atop j}\right)B_{2m+2-j}(-\frac\ell2)\cdot\frac{(\ell+1)^{j}}{j}
\\&&-\log\ell\cdot\left(B_{2m+2}(-\frac\ell2)+(1+\frac\ell2)B_{2m+1}(-\frac\ell2)\right)
\\&&+\frac{\log2\pi}2B_{2m+1}(-\frac\ell2)+\sum_{k=1}^{m}\frac{(-1)^{k+1}(2m+1)!}{2(2m+1-2k)!(2\pi)^{2k}}\zeta(2k+1)B_{2m+1-2k}(-\frac\ell2)
\\&&+\sum_{j=1-\M}^{2m}\sum_{q=\max\{0,1+j\}}^{2m+1}\left(\sum_{r=\max\{0,1+j\}}^q\left({q\atop r}\right)\frac{(-\frac12)^r}{r-j}\right)
\\&&\cdot\left({2m+1\atop q}\right)B_{2m+1-q}\cdot\left((-1)^{q-j}\frac{B_{q-j+1}}{q-j+1}-1+\frac1{q-j+1}\right)(-1)^j\ell^j
\\&&+\sum_{q=0}^{2m+1}\sum_{r=0}^{2m+1-q}\frac{(2m+1)!(q-1+\M)!}{q!(2m+1-q-r)!(q+r+\M)!}B_{2m+1-q-r}
\\&&\cdot\frac{(-1)^\M\ell^{-\M}}{2^q}
R(1,-q-r-\M,2,r,\ell)
\Bigg)
\\&=&\sum_{m=0}^{\lfloor n/2\rfloor}
\frac{c_1(E)^{n-2m}(c_1(E)^2-4c_2(E))^m}{(n-2m)!(2m+1)!}(-1)^{n}
\\&&\cdot\Bigg(
\frac{1}{2^{n+1}}\ell^{n+1}\log\frac{\ell}{2\pi}
+\frac{2m+1}{2^{n-1}3}\ell^{n}\log\ell
\\&&+\frac{7-4m   -6(2m+1)\log(2\pi)+24(2m+1)\zeta'(-1)}{2^{n+2}3}\ell^{n}
\Bigg)
\\&&
+O(\ell^{n-1}\log\ell).
\end{eqnarray*}
\end{theor}
To keep the full formula somewhat shorter it was written in terms of Bernoulli polynomials of $\ell$ instead of powers of $\ell$, using the equations (\ref{T1einfach}), (\ref{T3einfach}), (\ref{T21einfach}), (\ref{T22einfach}), (\ref{T23einfach}), (\ref{T24einfach}) in the following proof. An expression as a linear combination of $\ell^j$ can be obtained by using (\ref{T1voll}), (\ref{T3voll}), (\ref{T21voll}), (\ref{T22voll}), (\ref{T23voll}), (\ref{T24voll}) instead. The terms in the result are obtained by adding $T_1+T_{23}^{\rm odd}$, $T_3$, $T_{21}$, $T_{22}$, $T_{23}^{\rm even}$, $T_{24}$, $T_{25}$ in that order.

The asymptotic expansion of $T_\pi(\mtr{{\cal O}(\ell)})^{[0]}$ up to $O(\ell^{-1})$ was already given by Finski in \cite[Section 4.3]{Fi} using the formula for $T_\pi(\mtr{{\cal O}(\ell)})^{[0]}$ from \cite{K2}.
\begin{proof}
By \cite[Th. 1.1]{KSgenus}, the $X$-equivariant torsion is given by $T_\pi(\mtr{{\cal O}(\ell)})=e^{-\frac\ell2c_1(E)}T_\ell(c_1(E)^2-4c_2(E))$ where $T_\ell(-t^2)$ is the power series given by the same formula as $T_{tX}$ in Eq.\ (\ref{gEquivTorsion}), which splits naturally into three terms
\begin{eqnarray}
T_\ell(-t^2)&=&
-\frac{\cos\frac{(\ell+1)t}2}{\sin\frac{t}2}\sum_{m\geq1\atop m{\rm\,odd}}\left(2\zeta'(-m)
+{\cal H}_m\zeta(-m)\right)\frac{(-1)^{\frac{m+1}2}t^m}{m!}
\nonumber
\\&&+\sum_{m=1}^{|\ell+1|}\frac{\sin(2m-|\ell+1|)\frac{t}2}{\sin\frac{t}2}\log m
+\left(
\frac{\cos\frac{(\ell+1)t}2)}{t\sin\frac{t}2}
\right)^\#
\label{FormelTl}
\\&=:&\sum_{m=0}^\infty\frac{2(-1)^mt^{2m}}{(2m+1)!}T_1+\sum_{m=0}^\infty\frac{2(-1)^mt^{2m}}{(2m+1)!}T_2+\sum_{m=0}^\infty\frac{2(-1)^mt^{2m}}{(2m+1)!}T_3,
\end{eqnarray}
where $(t^{2m})^\#:=t^{2m}\cdot\left\{
{2{\cal H}_{2m+1}-{\cal H}_{m}\atop0}
\mbox{ if }{m\geq0,\atop m=-1}
\right.$ and $T_1,T_2,T_3$ depend on $m$ and $\ell$, but not on $t$.
Assume $\ell\geq-1$. By the definition of Bernoulli polynomials and the formula $B_m(x)=-m\zeta(1-m,x)$ ($m\in\BZ^+$, \cite[1.11(11)]{Erd1}) one gets the Taylor series
\begin{eqnarray}\label{DefBernoulli}
\frac{\cos\frac{(\ell+1)t}2}{t\sin\frac{t}2}&=&\sum_{m=0}^\infty B_{2m}(\frac\ell2+1)\frac{2(-1)^mt^{2m-2}}{(2m)!}
=\sum_{m=0}^\infty B_{2m}\left(-\frac\ell2\right)\frac{2(-1)^mt^{2m-2}}{(2m)!}
.
\end{eqnarray}
Therefore
\begin{eqnarray}\label{T3einfach}
T_3&=&-B_{2m+2}\left(-\frac\ell2\right)\frac{2{\cal H}_{2m+1}-{\cal H}_{m}}{2m+2}
\\&=&-(2{\cal H}_{2m+1}-{\cal H}_{m})\sum_{j=0}^{2m+2}\frac{(2m+1)!}{(2m+2-j)!j!}B_{2m+2-j}\cdot\left(-\frac\ell2\right)^j.
\label{T3voll}
\end{eqnarray}
Similarly
\begin{eqnarray}
\nonumber
T_1&=&\frac{(2m+1)!}{2(-1)^m}\cdot\Bigg(\sum_{r=0}^\infty B_{2r}\left(-\frac\ell2\right)\frac{2(-1)^rt^{2r-1}}{(2r)!}
\\&&\cdot\sum_{k=0}^\infty\left(2\zeta'(-2k-1)
+{\cal H}_{2k+1}\zeta(-2k-1)\right)\frac{(-1)^{k}t^{2k+1}}{(2k+1)!}\Bigg)^{[\deg_t=2m]}
\nonumber
\\&=&\sum_{k=0}^m
\left(2\zeta'(-2k-1)
-{\cal H}_{2k+1}\frac{B_{2k+2}}{2k+2}\right)\left({2m+1\atop2k+1}\right)B_{2m-2k}\left(-\frac\ell2\right)
\label{T1einfach}
\\&=&\sum_{k=0}^m\sum_{j=0}^{2m-2k}
\left(2\zeta'(-2k-1)
-{\cal H}_{2k+1}\frac{B_{2k+2}}{2k+2}\right)\frac{(2m+1)!B_{2m-2k-j}\cdot\left(-\frac\ell2\right)^j}{(2m-2k-j)!j!(2k+1)!}
\nonumber
\\&=&\sum_{j=0}^{2m}\sum_{k=0}^{\lfloor m-j/2\rfloor}
\left(2\zeta'(-2k-1)
-{\cal H}_{2k+1}\frac{B_{2k+2}}{2k+2}\right)\frac{(2m+1)!B_{2m-2k-j}\cdot\left(-\frac\ell2\right)^j}{(2m-2k-j)!j!(2k+1)!}.
\label{T1voll}
\end{eqnarray}
Using again the definition of Bernoulli polynomials, $T_2$ takes the form
\[
T_2=\sum_{k=1}^{\ell+1}B_{2m+1}\left(k-\frac\ell2\right)\log k.
\]
The polynomial $k\mapsto B_{2m+1}\left(k-\frac\ell2\right)$ is skew-symmetric around $k=\frac{\ell+1}2$. Hence we obtain by Propositions \ref{PHatStern} and \ref{AsympMitZeta}
\begin{eqnarray*}
\sum_{k=1}^{\ell+1}B_{2m+1}\left(k-\frac\ell2\right)\log k&=&
-2B_{2m+1}\left(k-\frac\ell2\right)^*(\ell+1)
\\&&+\log \ell\cdot\bz\left(B_{2m+1}\left(-k-\frac\ell2\right)\right)-\bzs\left(B_{2m+1}\left(k-\frac\ell2\right)\right)
\\&&+
\bz\left(B_{2m+1}\left(-k-\frac\ell2\right)\log^\ddagger\left(\frac{k+1}{\ell}+1\right)\right)^{[\deg_\ell>-\M]}
\\&&+\bz_{1,\ell,\M}^R\left(B_{2m+1}\left(-k-\frac\ell2\right)\right)
\\&=:&T_{21}+T_{22}+T_{23}+T_{24}+T_{25}.
\end{eqnarray*}
Using $B_{m}(x+y)=\sum_{h=0}^m\left({m\atop h}\right)B_{m-h}(x)y^h$ \cite[1.13(8)]{Erd1} and Proposition \ref{PHatStern} one obtains
\begin{eqnarray}
T_{21}&=&-\sum_{h=0}^{2m+1}\left({2m+1\atop h}\right)B_{2m+1-h}(-\frac\ell2)\cdot\frac{(\ell+1)^{h+1}}{(h+1)^2}
\nonumber
\\&=&-\sum_{h=1}^{2m+2}\left({2m+2\atop h}\right)B_{2m+2-h}(-\frac\ell2)\cdot\frac{(\ell+1)^{h}}{h(2m+2)}
\label{T21einfach}
\\&=&-\sum_{h=1}^{2m+2}\sum_{r=0}^{2m+2-h}\sum_{j=0}^{h}
\left({2m+2\atop r,2m+2-h-r,j,h-j}\right)\frac{B_{2m+2-h-r}}{h(2m+2)}\cdot(-\frac12)^r\ell^{r+j}
\nonumber
\\&=&-\sum_{r=0}^{2m+2}\sum_{h=1}^{2m+2-r}\sum_{j=0}^h
\left({2m+2\atop r,2m+2-h-r,j,h-j}\right)\frac{B_{2m+2-h-r}}{h(2m+2)}\cdot(-\frac12)^r\ell^{r+j}
\nonumber
\\&=&-\sum_{r=0}^{2m+2}\sum_{j=0}^{2m+2-r}\sum_{h=\max\{1,j\}}^{2m+2-r}
\left({2m+2\atop r,2m+2-h-r,j,h-j}\right)\frac{B_{2m+2-h-r}}{h(2m+2)}\cdot(-\frac12)^r\ell^{r+j}
\nonumber
\\&\stackrel{j\mapsto j-r\atop h\mapsto h-r}=&-\sum_{r=0}^{2m+2}\sum_{j=r}^{2m+2}\sum_{h=\max\{1+r,j\}}^{2m+2}
\left({2m+2\atop r,2m+2-h,j-r,h-j}\right)\frac{B_{2m+2-h}}{(h-r)(2m+2)}\cdot(-\frac12)^r\ell^j
\nonumber
\\&=&-\sum_{r=0}^{2m+2}\sum_{j=r}^{2m+2}\sum_{h=\max\{r+1,j\}}^{2m+2}
\left({2m+2\atop r,2m+2-h,j-r,h-j}\right)\frac{B_{2m+2-h}}{(h-r)(2m+2)}\cdot(-\frac12)^r\ell^j
\nonumber
\\&=&-\sum_{j=0}^{2m+2}\sum_{r=0}^j\sum_{h=\max\{r+1,j\}}^{2m+2}
\left({2m+2\atop r,2m+2-h,j-r,h-j}\right)\frac{B_{2m+2-h}}{(h-r)(2m+2)}\cdot(-\frac12)^r\ell^j
\nonumber
\\&=&-\sum_{j=0}^{2m+1}\sum_{h=j+1}^{2m+2}
\left(\sum_{r=0}^j\left({j\atop r}\right)\cdot\frac{(-\frac12)^r}{h-r}\right)
\left({2m+2\atop j,2m+2-h,h-j}\right)\frac{B_{2m+2-h}}{2m+2}\cdot\ell^j
\nonumber
\\&&-\sum_{j=1}^{2m+2}
\left(\sum_{r=0}^{j-1}\left({j\atop r}\right)\cdot\frac{(-\frac12)^r}{j-r}\right)
\left({2m+2\atop j}\right)\frac{B_{2m+2-j}}{2m+2}\cdot\ell^j
\nonumber
\end{eqnarray}
where the second sum in the last step corresponds to the case $h=j$. Applying $\sum_{r=0}^{j-1}\left({j\atop r}\right)\cdot\frac{(-\frac12)^{r-j}}{j-r}=\sum_{r=1}^j\frac{(-1)^r-1}r=-2{\cal H}_j+{\cal H}_{\lfloor j/2\rfloor}$ (\cite[p.\ 666]{KK}) and $\sum_{r=0}^{2m+1}\left({2m+1\atop r}\right)\cdot\frac{(-\frac12)^r}{2m+2-r}=\sum_{r=0}^{2m+1}\left({2m+2\atop r}\right)\cdot\frac{(-\frac12)^r}{2m+2}=0$ we get
\begin{eqnarray}
T_{21}&=&-\sum_{j=0}^{2m}\sum_{h=j+1}^{2m+2}
\left(\sum_{r=0}^j\left({j\atop r}\right)\cdot\frac{(-\frac12)^r}{h-r}\right)
\left({2m+2\atop j,2m+2-h,h-j}\right)\frac{B_{2m+2-h}}{2m+2}\cdot\ell^j
\nonumber
\\&&+\sum_{j=1}^{2m+2}
\left(2{\cal H}_j-{\cal H}_{\lfloor j/2\rfloor}\right)
\left({2m+2\atop j}\right)\frac{B_{2m+2-j}}{2m+2}\cdot(-\frac\ell2)^j.
\end{eqnarray}
When looking at top degree terms, replacing the sum over $r$ using Proposition \ref{nichtHa} provides the more suitable formula
\begin{eqnarray}
T_{21}&=&-\sum_{j=0}^{2m}\sum_{h=j+1}^{2m+2}
\left((-1)^j+\sum_{r=0}^{h-j-1}\left({j+r\atop r}\right)2^r\right)
\left({2m+2\atop h}\right)\frac{2^{-h}B_{2m+2-h}}{(h-j)(2m+2)}\cdot\ell^j
\nonumber
\\&&+\sum_{j=1}^{2m+2}
\left(2{\cal H}_j-{\cal H}_{\lfloor j/2\rfloor}\right)
\left({2m+2\atop j}\right)\frac{B_{2m+2-j}}{2m+2}\cdot(-\frac\ell2)^j.
\label{T21voll}
\end{eqnarray}
Using the formula for $B_{m}(x+y)$ again one also finds
\begin{eqnarray}
T_{23}&=&-\sum_{h=0}^{2m+1}\left({2m+1\atop h}\right)\bzs B_{2m+1-h}(k)\cdot(-\frac\ell2)^h
\label{T23voll}
\\&=&-\sum_{k=0}^{2m+1}\left({2m+1\atop k}\right)B_{2m+1-k}(-\frac\ell2)\zeta'(-k).
\nonumber
\end{eqnarray}
As $\zeta'(-2k)=(-1)^k\frac{(2k)!}{2(2\pi)^{2k}}\zeta(2k+1)$ for $k\in\BZ^+$ and $\zeta'(0)=-\log\sqrt{2\pi}$, the part of this sum for $k$ even equals
\begin{eqnarray}
T_{23}^{\rm even}=\frac{\log2\pi}2B_{2m+1}(-\frac\ell2)+\sum_{k=1}^{m}\frac{(-1)^{k+1}(2m+1)!}{2(2m+1-2k)!(2\pi)^{2k}}\zeta(2k+1)B_{2m+1-2k}(-\frac\ell2).
\label{T23einfach}
\end{eqnarray}
The odd part cancels with one half of the $\zeta'$-summand in $T_1$ as in Eq.\ (\ref{T1einfach}):
\begin{eqnarray*}
T_{23}^{\rm odd}&=&-\sum_{k=0}^{m}\left({2m+1\atop2k+1}\right)B_{2m-2k}(-\frac\ell2)\zeta'(-2k-1).
\end{eqnarray*}
We use the equation $B_m=(-1)^{m-1}m\zeta(1-m)$ and Euler's convolution identity $\forall m\in\BZ^+:\sum_{j=0}^m\left({m\atop j}\right)B_jB_{m-j}=-mB_{m-1}-(m-1)B_m$ \cite[p.\ 42 Eq.\ (13)]{Nielsen} to find the identity
\begin{eqnarray}\label{zetaBernoulli}\nonumber
\bz B_{m}(-k)&=&\sum_{j=0}^m\left({m\atop j}\right)B_j\cdot(-1)^{m-j}\zeta(j-m)
=\sum_{j=0}^m\left({m\atop j}\right)B_j\frac{B_{1+m-j}}{1+m-j}
\\&=&\frac1{m+1}\left(\sum_{j=0}^{m+1}\left({m+1\atop j}\right)B_jB_{1+m-j}-B_{m+1}\right)
=-B_{m+1}-B_m.
\end{eqnarray}
Thus
\begin{eqnarray}
T_{22}&=&\log\ell\cdot\sum_{j=0}^{2m+1}\left({2m+1\atop j}\right)\bz B_{2m+1-j}(-k)\cdot(-\frac\ell2)^j
\nonumber
\\&\stackrel{(\ref{zetaBernoulli})}=&-\log\ell\cdot\sum_{j=0}^{2m+1}\left({2m+1\atop j}\right)(B_{2m+2-j}+B_{2m+1-j})\cdot(-\frac\ell2)^j
\label{T22voll}
\\&=&-\log\ell\cdot\Bigg(\sum_{j=0}^{2m+2}\frac{2m+2-j}{2m+2}\left({2m+2\atop j}\right)B_{2m+2-j}\cdot(-\frac\ell2)^j+B_{2m+1}(-\frac\ell2)\Bigg)
\nonumber
\\&=&-\log\ell\cdot\big(B_{2m+2}(-\frac\ell2)+(1+\frac\ell2)B_{2m+1}(-\frac\ell2)\big).
\label{T22einfach}
\end{eqnarray}
Furthermore using
\begin{eqnarray}\label{zeta+1}
\bz((k+1)^n)\stackrel{(\ref{MultLogv0})}=\Phi(1,-n,2)+\frac1{n+1}=\zeta(-n)-1+\frac1{n+1}=\frac{(-1)^n}{n+1}B_{n+1}-1+\frac1{n+1}
\end{eqnarray}
one finds
\begin{eqnarray}
T_{24}&=&\bz\Bigg(\sum_{r=0}^{2m+1}\left({2m+1\atop r}\right)B_{2m+1-r}(-k)\cdot(-\frac\ell2)^r
\sum_{j=1}^\infty\frac{(-1)^{j+1}}j(\frac{k+1}\ell)^j\Bigg)^{[\deg_\ell>-\M]}
\nonumber
\\&=&\bz\sum_{j=1-\M}^{2m}\sum_{r=\max\{0,1+j\}}^{2m+1}\left({2m+1\atop r}\right)B_{2m+1-r}(-k)\cdot(k+1)^{r-j}\frac{(-1)^{j+1}\ell^j}{2^r(r-j)}
\nonumber
\\&=&\bz\sum_{j=1-\M}^{2m}\sum_{r=\max\{0,1+j\}}^{2m+1}\left({2m+1\atop r}\right)(-1)^{2m+1-r}B_{2m+1-r}(k+1)\cdot(k+1)^{r-j}\frac{(-1)^{j+1}\ell^j}{2^r(r-j)}
\nonumber
\\&=&\bz\sum_{j=1-\M}^{2m}\sum_{r=\max\{0,1+j\}}^{2m+1}\sum_{q=0}^{2m+1-r}\left({2m+1\atop r,q,2m+1-r-q}\right)B_{2m+1-r-q}\cdot(k+1)^{q+r-j}\frac{(-1)^{j+r}\ell^j}{2^r(r-j)}
\nonumber
\\&=&\bz\sum_{j=1-\M}^{2m}\sum_{r=\max\{0,1+j\}}^{2m+1}\sum_{q=r}^{2m+1}\left({2m+1\atop r,q-r,2m+1-q}\right)B_{2m+1-q}\cdot(k+1)^{q-j}\frac{(-1)^{j+r}\ell^j}{2^r(r-j)}
\nonumber
\\&=&\bz\sum_{j=1-\M}^{2m}\sum_{q=\max\{0,1+j\}}^{2m+1}\sum_{r=\max\{0,1+j\}}^q\left({2m+1\atop r,q-r,2m+1-q}\right)B_{2m+1-q}\cdot(k+1)^{q-j}\frac{(-1)^{j+r}\ell^j}{2^r(r-j)}
\nonumber
\\&=&\bz\sum_{j=1-\M}^{2m}\sum_{q=\max\{0,1+j\}}^{2m+1}\left(\sum_{r=\max\{0,1+j\}}^q\left({q\atop r}\right)\frac{(-\frac12)^r}{r-j}\right) \left({2m+1\atop q}\right)B_{2m+1-q}\cdot(k+1)^{q-j}(-\ell)^j
\nonumber
\\&\stackrel{(\ref{zeta+1})}=&\sum_{j=1-\M}^{2m}\sum_{q=\max\{0,1+j\}}^{2m+1}\left(\sum_{r=\max\{0,1+j\}}^q\left({q\atop r}\right)\frac{(-\frac12)^r}{r-j}\right)\cdot\left({2m+1\atop q}\right)
\nonumber
\\&&\cdot B_{2m+1-q}\cdot\left((-1)^{q-j}\frac{B_{q-j+1}}{q-j+1}-1+\frac1{q-j+1}\right)(-\ell)^j.
\label{T24einfach}
\end{eqnarray}
In the case $j<0$ Proposition \ref{nichtHa}, more specifically Eq.\ (\ref{nichtHa2}), provides again a different expression for the sum over $r$ as
\begin{eqnarray}
T_{24}&\stackrel{j<0}=&\sum_{j=1-\M}^{2m}\sum_{q=0}^{2m+1}\left(
\frac{1-2^{-q-1}\sum_{r=0}^{-j-1}\left({q+r\atop q}\right)2^{-r}}{-j\left({q-j\atop -j}\right)}
\right)\cdot\left({2m+1\atop q}\right)
\nonumber
\\&&\cdot B_{2m+1-q}\cdot\left((-1)^{q-j}\frac{B_{q-j+1}}{q-j+1}-1+\frac1{q-j+1}\right)(-\frac\ell2)^j.
\label{T24voll}
\end{eqnarray}
Finally one obtains
\begin{eqnarray}
T_{25}&=&\bz^R_{1,\ell,\M}
\sum_{q=0}^{2m+1}\left({2m+1\atop q}\right)B_{2m+1-q}(-k)\cdot\left(-\frac\ell2\right)^q
\nonumber
\\&=&\bz^R_{1,\ell,\M}
\sum_{q=0}^{2m+1}\sum_{r=0}^{2m+1-q}\left({2m+1\atop q,r,2m+1-q-r}\right)B_{2m+1-q-r}\cdot\frac{(-1)^{r+q}}{2^q}k^r\ell^q
\nonumber
\\&=&\sum_{q=0}^{2m+1}\sum_{r=0}^{2m+1-q}\left({2m+1\atop q,r,2m+1-q-r}\right)B_{2m+1-q-r}
\nonumber
\\&&\cdot\frac{(-1)^{r+q}}{2^q}
\frac{(-1)^{r+q+\M}\ell^{-\M}}{\left({q+r+\M\atop r}\right)\cdot(q+\M)}R(1,-q-r-\M,2,r,\ell)
\nonumber
\\&=&\sum_{q=0}^{2m+1}\sum_{r=0}^{2m+1-q}\frac{(2m+1)!(q-1+\M)!}{q!(2m+1-q-r)!(q+r+\M)!}B_{2m+1-q-r}
\nonumber
\\&&\cdot\frac{(-1)^\M\ell^{-\M}}{2^q}
R(1,-q-r-\M,2,r,\ell).
\end{eqnarray}

By (\ref{T21voll}) and (\ref{T1einfach}) the coefficients of $\ell^{2m+2}$ and $\ell^{2m+1}$ in $T_{21}$ cancel with the corresponding terms of $T_3$, as $2{\cal H}_{2m+2}-{\cal H}_{m+1}=2{\cal H}_{2m+1}-{\cal H}_{m}$. Furthermore by (\ref{T22voll}) and (\ref{T23einfach}) the top terms of $T_{22}$ and $T_{23}$ yield
$$
(T_{22}+T_{23})^{[\deg_\ell=2m+1]}=2^{-2m-2}\log\frac{\ell}{2\pi}.
$$
The coefficient of $l^{2m}$ is given by
\begin{eqnarray*}
&&(T_1+T_3+T_{21}+T_{22}+T_{23}+T_{24})^{[\deg_\ell=2m]}=
(\frac{2m+1}{2^{2m-1}}\zeta'(-1)-\frac{2m+1}{3\cdot2^{2m+2}})
\\&&-\frac{2m+1}{3\cdot 2^{2m+2}}(2{\cal H}_{2m+1}-{\cal H}_m)
+(\frac{1}{2^{2m+1}}-\frac1{2^{2m+2}}+(2{\cal H}_{2m}-{\cal H}_m)\frac{2m+1}{3\cdot2^{2m+2}})
\\&&+\frac{2m+1}{3\cdot2^{2m}}\log\ell
-(\frac{2m+1}{2^{2m}}\zeta'(-1)+\frac{2m+1}{2^{2m+2}}\log2\pi)
+\frac{7}{3\cdot2^{2m+3}}
\\&=&\frac1{3\cdot2^{2m+3}}\Big(
7-4m+(2+4m)\log\frac{\ell^4}{(2\pi)^3}+24(1+2m)\zeta'(-1)
\Big).
\end{eqnarray*}
Replacing $(-t^2)$ by $c_1(E)^2-4c_2(E)$ shows the claimed formula as
\begin{align*}
T_\pi(\mtr{{\cal O}(\ell)})^{[2n]}&=\left(e^{-\frac\ell2c_1(E)}T_\ell(c_1(E)^2-4c_2(E))\right)^{[2n]}
\\&=\sum_{r\geq0\atop r+2m=n}\frac1{r!}(-\frac\ell2c_1(E))^r\sum_{m=0}^{\lfloor n/2\rfloor}\frac{2(c_1(E)^2-4c_2(E))^m}{(2m+1)!}(T_1+T_2+T_3).\qedhere
\end{align*}
\end{proof}
\begin{Bem}Further coefficients of $T_1+T_3+T_{21}+T_{22}+T_{23}+T_{24}$, multiplied by the corresponding power of $\ell$, are given by
$$
\frac{l^{2m-1}}{3\cdot 2^{2m+2}}\Big(
5+18m-8m^2+8m(1+2m) (\log\frac{\ell}{\sqrt{2\pi}}-\frac3{2 \pi^2} \zeta(3)+6 \zeta'(-1))
\Big)
$$
and
\begin{eqnarray*}
&&\quad
\frac{l^{2m-2}}{2^{2m+3}}\Big(\frac1{135}\big(
-87 - 121 m + 840 m^2 - 152 m^3\big)
\\&&+4m(4 m^2 - 1) \big(
\frac{4}{45}\log\ell
-\frac{1}{\pi^2} \zeta(3)
+\frac{8}3 \zeta'(-3) +\frac{4}3\zeta'(-1)
\big)
\Big).
\end{eqnarray*}
\end{Bem}
\begin{Bem}
By \cite[Th. 11.2]{KSgenus}, the $R$-class term in the arithmetic Riemann--Roch is given by the first summand in Eq.\ (\ref{FormelTl}), corresponding to $T_1$. Hence the asymptotic of the difference of the torsion form and the $R$-class term is obtained by omitting the term given in Eq.\ (\ref{T1einfach}) at the right hand side in Th.\ \ref{asympTpi}. This term is of order $O(\ell^{2m})$.
\end{Bem}
Finally we verify that the highest degree term in Th.\ \ref{asympTpi} corresponds to the formula given by Puchol in \cite[Th. 1.3]{Puchol}:
\begin{lemma}\label{VglPuchol}  Using the curvature $\Omega^{{\cal O}(1)}$ of the ${\cal O}(1)$ bundle over the total space of the $\BP^1\BC$-fibration, the $2n$-component of the torsion form equals
\begin{eqnarray*}
T_\pi(\mtr{{\cal O}(\ell)})^{[2n]}&=&\frac12\left(
\int_{\BP^1\BC}\log\frac{\ell \mathring \Omega^{{\cal O}(1)}}{2\pi}e^{-\frac\ell{2\pi i}\Omega^{{\cal O}(1)}}
\right)^{[2n]}+o(\ell^{n+1}).
\end{eqnarray*}
\end{lemma}
Here $\mathring \Omega^{{\cal O}(1)}$ is defined as in Th.\ \ref{BVFinski}.
\begin{proof}
In this case $\mathring \Omega^{{\cal O}(1)}={\rm id}_\BC$.
We denote the fibre by $Z$ and the vertical tangent space by $T\pi$. \cite[Prop.\ 11.1]{KSgenus} states that
$$
\pi^*c_1(E)=c_1(T\pi)-2c_1({\cal O}(1))\quad\mbox{and}\quad\pi^*(c_1(E)^2-4c_2(E))=c_1(T\pi)^2.
$$
Therefore using the projection formula and $\int_{\BP^1\BC}c_1(T\pi)=2$ we obtain
\begin{eqnarray*}
\lefteqn{
\frac{\ell^{n+1}}2\log\frac{\ell}{2\pi}\cdot
\int_Z \frac{c_1({\cal O}(1))^{n+1}}{(n+1)!}
}\\&=&
\frac{\ell^{n+1}}2\log\frac{\ell}{2\pi}\cdot
\int_Z\sum_{j=0}^{n+1}\frac1{2^n(n+1)!}\left({n+1\atop j}\right)c_1(T\pi)^j(-\pi^*c_1(E))^{n+1-j}
\\&\stackrel{m:=2j+1}=&
\frac{\ell^{n+1}}2\log\frac{\ell}{2\pi}\cdot
\int_Z\sum_{m=0}^{\lfloor n/2\rfloor}\frac{(-1)^{n-2m}}{2^{n+1}(2m+1)!(n-2m)!}
\\&&\cdot c_1(T\pi)\pi^*((c_1(E)^2-4c_2(E))^mc_1(E)^{n-2m})
\\&=&
\ell^{n+1}\log\frac{\ell}{2\pi}\cdot
\sum_{m=0}^{\lfloor n/2\rfloor}\frac{(-1)^{n}}{2^{n+1}(2m+1)!(n-2m)!}(c_1(E)^2-4c_2(E))^mc_1(E)^{n-2m}
\end{eqnarray*}
which is equal to the top most term in the last formula in Theorem \ref{asympTpi}.
\end{proof}

\end{document}